\documentclass[a4paper,oneside,11pt]{article}
\usepackage{makeidx}
\usepackage{amssymb}
\usepackage{amsmath}
\usepackage{amsthm}
\usepackage[all]{xy}
\usepackage{graphics}
\usepackage{graphicx}
\def\commutatif{\ar@{}[rd]|{\circlearrowleft}}
\def \gog {\mathfrak g}

\def \goh {\mathfrak h}
\def \gosl {{\mathfrak s}{\mathfrak l}}
\def \goso {{\mathfrak s}{\mathfrak o}}
\def \gosp {{\mathfrak s}{\mathfrak p}}

\newcommand{\cro}[2]{\bigl[ #1 , #2 \bigr]}
\newcommand{\croq}[2]{\bigl[ #1 , #2 \bigr]_q}
\newcommand{\crc}[1]{\bigl[ #1 \bigr]}
\newcommand{\crq}[1]{\bigl[ #1 \bigr]_q}
\newcommand{\crqq}[1]{\bigl[ #1 \bigr]_{q^2}}

\makeindex
\usepackage{fullpage}
\usepackage{mathrsfs}
\makeatletter

\newtheorem{Def}{Definition}[section]

\newtheorem{Theo}[Def]{Theorem}

\newtheorem{Prop}[Def]{Proposition}
\newtheorem{Cor}[Def]{Corollary}
\newtheorem{Lem}[Def]{Lemma}
\newenvironment{Rem}{\textbf{Remarks} \ \ \ }{\newline}
\renewenvironment{proof}{\textbf{proof}}{\qed \newline}

\title{Quantization of coisotropic subalgebras in complex semisimple Lie bialgebras}
\author{Jonathan Ohayon \\
Universit\'e Montpellier II, I3M,  \\
  Place eug\`ene bataillon, 34095 Montpellier,  FRANCE
 \vspace{1mm}  \\
E-mail: johayon@math.univ-montp2.fr }

\begin{document}\normalsize

\maketitle

\begin{abstract}
The aim of this article is to give a quantization of some coisotropic subalgebras in complex semisimple Lie bialgebras. The coisotropic subalgebras that will be quantized are those given by Zambon in his paper "`A Construction for coisotropic subalgebras of Lie Bialgebras"' \cite{[Za1]}. We will also extend the construction for the exceptional complex semisimple Lie bialgebras.
\end{abstract}

\section{Introduction}

Since their introduction in 1986 by V. Drinfeld \cite{[Dr1]}, quantum groups arise as Hopf algebras neither commutative nor cocommutative. They play a central role in the deformation theory of Lie algebras but also of many others structures. One of the first problem to arise was the deformation of the Lie bialgebras \cite{[Dr3]}, which was connected with the deformation of the Poisson-Lie group by the V. Drinfeld functor between $U_h(\gog)$ (QUEA) and $F\left[\left[G\right]\right]$ (QFSHA), where $G$ is a Poisson-Lie group and $\gog = Lie(G)$ is a Lie bialgebra. This duality between the two structures was introduced by V. Drinfeld \cite{[Dr3]}, and was, later on, deepen by F. Gavarini \cite{[Ga1]}. P. Etingof and J. Kazhdan proved that all Lie bialgebras can be quantized \cite{[EK1]}. The remaining questions lie in how the different structures that can be found in the Lie bialgebras can be translated in their quantum counterpart. One of such structure is the coisotropic subalgebra.
\begin{Def} 
A coisotropic subalgebra $\goh$ of a Lie bialgebra $(\gog,\left[ , \right],\delta)$ is a Lie subalgebra which is also a Lie coideal, meaning that $\delta(\goh) \subset \goh \wedge \gog$ .
\end{Def}
This problem of quantization was studied by N. Ciccoli in his article "Quantization of Coisotropic Subgroups" \cite{[Ci1]}. But  as there is a duality between the Poisson-Lie group and the Lie bialgebras, there is one between the coisotropic subalgebra $\goh$ of a Lie bialgebra $\gog$ and the coisotropic subgroup H of a Poisson-Lie group G. This duality can even be extended as we can associate an homogeneous G-space G/H  in a formal sense to the coisotropic subgroup H of a Poisson-Lie group G. This give us four different approaches to the quantization of the coisotropic subgroups by using the quantum duality principle given by N. Ciccoli and F. Gavarini in their article \cite{[CG1]}.
The quantization problem of these objects is still open. It is interresting to note that an answer was given in the special case where the coisotropic subalgebra is a sub Lie bialgebra. This is in fact given by the functoriality of the quantization functor of P. Etingof and J. Kazhdan. Recently, M. Zambon has developed a method to construct coisotropic subalgebras of Lie bialgebras and has explicited this construction in the case of semisimple complex Lie bialgebras \cite{[Za1]}. Therefore, in the aim of giving an answer to the problem of quantization, it is interesting to look at this construction and see if it can be quantized in this case. \newline
Accordingly, the paper is organized as follows. In section 2, we will recall the method to construct the coisotropic subalgebras in semisimple complex Lie bialgebras, for which we will mainly give the results found by M. Zambon and  detail the main steps to follow in order to prove that the coisotropic subalgebras hence constructed can be quantized. In the following sections, we will construct the coisotropic subalgebras by using the Chevalley basis and Serre's relation. Then we will give their quantum counterpart and prove that they are indeed a quantization of the coisotropic subalgebras constructed, in the sense of N. Ciccoli and G. Gavarini. In the study, we will first construct and quantize the coisotropic subalgebra for the usual semi-simple complex Lie bialgebras classified by their types, first of type $A_n$ which corresponds to $\gosl(n+1)$, secondly of type $C_n$  which corresponds to $\gosp(2n)$, then of type $D_n$ which corresponds to $\goso(2n)$, and finally of type $B_n$ which corresponds to $\goso(2n+1)$.  Then, finally, we will repeat the process  for the exceptional semi-simple complex Lie bialgebras, classified by their types.

\subsection*{Acknowledgement}

I thank Gilles Halbout and Benjamin Enriquez for helpful conversation and remarks, that help complete this work.

\section{construction of coisotropic subalgebras in the semi-simple complex algebras}

In his paper \cite{[Za1]}, M. Zambon gives a construction for coisotropic subalgebras of Lie bialgebras and studies his example for the case of the semi-simple complex algebras $\gosl(n+1)$, $\goso(2n+1)$, $\gosp(2n)$ et $\goso(2n)$. First let us recall some of the main theorems that will give rise to those example. In the general case we will have the following:

\begin{Theo}  \label{Coiso}
Let $G$ be a Poisson Lie group corresponding to an r-matrix $\pi$, $X \in \gog=Lie(G)$, g:=exp(X). Assume that
\[ \crc{X , \crc{X , \pi}}=\lambda\crc{X , \pi} \ \ \ for \ some \ \lambda \in \mathbb{R}.\]
Then $\goh^g$ is a coisotropic subalgebra of $\gog$. Further
\[ \goh^g = \crc{X , \pi}^{\#}\gog^*. \]
where $\crc{X , \pi}^{\#}$ correspond to the map from $g^* \to g$ contracted with $\crc{X , \pi}$.
\end{Theo}
This theorem works for every Lie bialgebra and gives rise to coisotropic subalgebras of even dimension. But it is not giving all of them. One can wonder if there is a less restrictive condition that will give rise to all the coisotropic subalgebras.
\newline
Let's now, restrict ourselves to the case of $\gog$ a semi-simple complex Lie bialgebra. By using the roots system R of $\gog$, we can construct families of coisotropic subalgebras.  For $\alpha \in R^+$, the positive roots we have $\alpha = \alpha_{i_1} \cdots \alpha_{i_r}$ where $\alpha_{i_j} \in \{ \alpha_1 \cdots \alpha_n\}$, we can associate to $\alpha$ a non-zero element \cite{[FH91]},
\[ e_{\alpha} = \cro{\cro{e_{\alpha_{i_1}}}{e_{\alpha_{i_{2}}}}}{\dots,e_{\alpha_{i_r}}} \in \gog^{\alpha} \]
and in the same way we associate a non zero element to $- \alpha$:
\[ f_{\alpha} = \cro{\cro{f_{\alpha_{i_1}}}{f_{\alpha_{i_{2}}}}}{\dots,f_{\alpha_{i_r}}} \in \gog^{-\alpha} \]
Those elements will give rise to a r-matrix defined as follows:
\[ \pi:= \sum_{\alpha \in R^+}{\lambda_{\alpha} e_{\alpha} \wedge f_{\alpha}} \]
where $\lambda_{\alpha} = \frac{1}{K( e_{\alpha},f_{\alpha})}$ and K is the killing form (a non degenerative definite positive bilinear form) associated to the Lie bialgebra.
\begin{Lem} \label{lCo}
Let $X \in \gog$ and assume that for all $\alpha \in R^+$, we have:
\begin{enumerate}
\item $\cro{X}{\cro{X}{e_{\alpha}}} \wedge f_{\alpha} = 0$
\item $\cro{X}{e_{\alpha}} \wedge \cro{X}{f_{\alpha}} = 0$
\item $e_{\alpha} \wedge \cro{X}{\cro{X}{f_{\alpha}}} = 0$
\end{enumerate}
Then X satisfies the condition of theorem \ref{Coiso} with $\lambda = 0$.
\end{Lem}
\begin{Prop}  \label{pCo}
Let $\beta \in R^+$ satisfying the following condition: \newline
For all $\alpha \in R$:  ($\alpha + \mathbb{Z}\beta$) $\cap R$ does not contain a string of three consecutive elements.\newline
Then $e_{\beta}$ et $f_{\beta}$ satisfies lemma \ref{lCo} and by consequence theorem \ref{Coiso}.
\end{Prop}
\begin{Cor}  \label{CCo}
Assume that $\beta \in R^+$ satisfies the condition in the proposition \ref{pCo}. Let $\gog_{\mathbb{R}}$ denote $\gog$ viewed as a real Lie algebra. Then $\cro{e_{\beta}}{\pi}^{\#}\gog_{\mathbb{R}}^*$ and $\cro{f_{\beta}}{\pi}^{\#}\gog_{\mathbb{R}}^*$
\begin{itemize}
\item are coisotropic subalgebras of $\gog_{\mathbb{R}}$.
\item their complexification are coisotropic subalgebras of the complex Lie bialgebra $\gog$
\end{itemize}
\end{Cor}
We want to give a quantization to this construction. First, let's recall what we mean by quantization in this case. Like we said in the introduction, the problem of quantization of such object was studied by N. Ciccoli and F. Gavarini. In their paper \cite{[CG1]}, they gave a characterization of the quantization of the coisotropic subalgebras.
\begin{Def} 
A quantization of a coisotropic subalgebras $\goh$ of $\gog$ is a subalgebra, left (or right) coideal $B_h$ of $U_h(\gog)$ such that:
\[ B_h/hB_h \cong \pi_{U_h}(B_h) = U(\goh) \]
where $\pi_{U_h} : U_h(\gog) \rightarrow U(\gog)$ is the specialization map at $h=0$.
\end{Def}
The constraint $B_h/hB_h \cong \pi_{U_h}(B_h) = U(\goh)$ means the following. We have a map $B_h \to U_h(\gog) \to U_h(\gog)/hU_h(\gog) \cong  U(\gog)$ and the composed map  $B_h \to U(\gog)$ can be factored through $B_h/hB_h$.
\[ \xymatrix{
   B_h  \ar[r]^-{} \ar[d]_-{\pi_{B_h}}  &  U(\gog)  \\
   B_h/hB_h \ar[ru]^-{} } \]
Then we want the factored map $B_h/hB_h \to U(\gog)$ to be a bijection in $\pi_{U_h}(B_h)$ which should coincide with $U(\goh)$. \newline
They also demonstrated that this constraint can be replaced by $B_h \cap h U_h(\gog) = h B_h$. Indeed we have that $\pi_{U_h}(B_h) = B_h / (B_h \cap hU_h(\gog))$ and therefore $B_h/hB_h \cong \pi_{U_h}(B_h)$. \newline
\begin{Rem} It is easy to see that if we have a subalgebra left coideal $B_h$ of $U_h(\gog)$ such that $B_h/hB_h = U(\goh)$ then $\goh$ is a coisotropic subalgebra of $\gog$. Meaning that the semi-classical linit is still well defined in this context. 
\end{Rem} 

  We will now detail the steps that we will take in the rest of the paper.
In the following sections we will give a quantization of the different coisotropic subalgebra that we can construct using the preceding theorems and definitions. To do so, we first need to determine the roots $\beta$ that will satisfy the condition of proposition \ref{pCo}. Then, we need to fix a cartan in order to construct the r-matrix given by
\[ \pi:= \sum_{\alpha \in R^+}{\lambda_{\alpha} e_{\alpha} \wedge f_{\alpha}} \]
and finally we need to compute $\cro{e_{\beta}}{\pi}$ in order to determine the elements that will generate the coisotropic subalgebra $\goh$ according the corollary \ref{CCo}. \newline
In a second time, we will choose a candidate $B_h$ to be the quantization, which will be the algebra spanned by a lift up of the generators of the coisotropic subalgebra $U(\goh)$ in $U_h(\gog)$. We will then verify that it is a subalgebra, left (or right) coideal of the bialgebra $U_h(\gog)$. \newline
And finally we will need to check that it is indeed the quantization of $\goh$. Meaning that we have to verify if $B=B_h/hB_h$ is isomorphic to $U(\goh)$.  For that we will use a proof similar to the one of Poincare Birkhoff Witt theorem.\newline
We will prove that $S(\goh)$ is isomorphic as a vectorial space to $B$ which will give us the wanted isomorphism by using the Poincare Birkhoff Witt theorem. By construction we have that $U(\goh) \subset B$ therefore we directly have the injection of $S(\goh)$ in $B$. \newline
Therefore, only the surjectivity remains, to prove it we will use the following proposition, for which we need to chose an order in $B_h$. 
\begin{Prop} \label{PropQ}
All elements $A$ in $B_h$,can be written in the form $A=\sum_k \sum_n h^n X_{n1} \cdots X_{nk}$ where $X_{ni}$ are elements of $B_h$ of degree 1 and without $h$. If all monome $X = X_{n1} \cdots X_{nk}$ can be written in the form:
\[ X = Y + X' + h*X''\]
where $Y = X_{n\sigma(1)} \cdots X_{n\sigma(k)}$ is well ordered when considering the order chosen, $X'$ is an element of degree inferior to $k$ and $X''$ is an element in $B_h$.  Then $B=B_h/hB_h$ is isomorphic to $S(\goh)$.
\end{Prop}
\begin{Rem} 
Following the proof of  Poincare Birkhoff Witt theorem, this proposition will prove the surjectivity of $S(\goh)$ in $B$.
One can see that we only need to prove this proposition for elements of degree 2. Because by induction, we can extend it for elements of degree superior to 2. This is done by permuting the elements two by two.
\end{Rem}
Therefore we will use the following corollary, in wich we need to fix a set of generators in $B_h$, which will be the elements of degree 1.
\begin{Cor} \label{LemQ}
If for all $X_1$,$X_2$ generators of $B_h$, we have:
\[ X_1X_2 - X_2X_1 =  X' + hX'' \]
where X' is either a generator or 0 and X'' is in $B_h$. Then $B=B_h/hB_h$ is isomorphic to $S(\goh)$.
\end{Cor}
With this we can prove that $B_h$ is a left coideal of $U_h(\gog)$ and that $B=B_h/hB_h$ is isomorphic to $U(\goh)$ by using  corollary \ref{LemQ} and Poincare Birkhoff Witt theorem. Meaning that $B$ is isomorphic to $S(\goh)$ therefore $B$ is isomorphic as a vector space to $U(\goh)$ and so we have proved that $B_h / h B_h = (B,\Delta,\mu,S) = (U(\goh),\Delta,\mu,S)$.\newline

To work in the semi-simple complex Lie bialgebra, we need to consider the quantization of Jimbo and V. Drinfeld, which give rise to the $U_q(\gog)$.
\begin{Def}
Let $\gog$ be a finite-dimensional complex semi-simple Lie algebra with cartan matrix ($a_{ij}$). Then $U_q(\gog)$ is the associative algebra over $\mathbb{Q}(q)$ with generators $X_i^+$, $X_i^-$, $K_i$ and $K_i^{-1}$, $1 \le i \le n$, and the following relations:
\[ K_iK_j = K_jK_i, K_iK_i^{-1} = K_i^{-1}K_i=1, \]
\[ (1) \crc{K_i,X_i^+}_{d_i*a_{ij}} = K_iX_i^+ - q^{d_i*a_{ij}}X_i^+K_i = 0 , \ \ \crc{K_i,X_i^-}_{-d_i*a_{ij}} = K_iX_i^- - q^{-d_i*a_{ij}}X_i^+K_i = 0 ,\]
\[ (b) \crc{X_i^+,X_j^-} =  \partial_{ij} \frac{K_i - K_i^{-1}}{q^{d_i} -q^{-d_i}} , \]
\[ (c) \sum_{r=0}^{1-a_{ij}} (-1)^r\left[\stackrel{1-a_{ij}}{r}\right]_d^{d_i} (X^{\pm})^{1-a_{ij}-r}X_j^{\pm}(X_i^{\pm})^r = 0 \ \ \ if \ i \neq j \]
There is a Hopf Algebra structure on $U_q(\gog)$ with commultiplication $\Delta$ defined as follow:
\[ \Delta(K_i) = K_i \otimes K_i,\]
\[ \Delta(X_i^{+})= X_i^{+} \otimes K_i + 1 \otimes X_i^{+}, \ \ \Delta(X_i^{-})= X_i^{+} \otimes 1 + K_i^{-1} \otimes X_i^{-} \]
We also set for the rest of the paper, the following notation:
\[ \crc{A,B}_{q^k} = AB - q^k BA \]
\end{Def}
We can and we will in the following section identify $U_q(\gog)$ with $U_h(\gog)$. It is done by identifying $q$ with $e^{h}$ and $K_i$ with $e^{d_i h H_i}$. \newline
We need also to set some notations: 
\begin{itemize} 
\item We will note the $q^a$-bracket of two elements A, B by 
\[  \crc{A,B}_{q^a} = AB - q^a BA \]
\item We will say that two elements A and B $q^a$-commute if $\crc{A,B}_{q^a} = 0$
\end{itemize}

\begin{Rem}
In the following sections, we will only consider the coisotropic subalgebras $\cro{e_{\beta}}{\pi}^{\#}\gog^*$ as the demonstration for $\cro{f_{\beta}}{\pi}^{\#}\gog^*$ is identical. The only change is that the candidate $B_h$ will no longer be a left coideal like for the previous case but a right coideal.
\end{Rem}

\section{$\gosl(n+1)$}

Let $\gog = \gosl(n+1)$ with Cartan subalgebra given by the diagonal matrices. The roots' set of $\gog$ is $\{ L_i - L_j \} _{ \left( i \neq j \right) } \subset \mathbb{R}^{n+1}$. Following the protocol, we have to check which roots satisfy the assumption of proposition \ref{pCo}. It is easy to check that all the roots do.\newline
Next we have to determine the r-matrix needed in the construction.
\[ \pi = \sum_{\alpha \in R^+}{\lambda_{\alpha} e_{\alpha} \wedge f_{\alpha}} \]
for the root $\alpha = L_i - L_j$ we have the vector $e_{\alpha} = e_{ij}$ and $f_{\alpha} = e_{ji}$. Therefore we can compute the r-matrix $\pi$:
\[\pi = \lambda \sum_{i < j} e_{i,j} \wedge e_{j,i}\]
where $\lambda$ is a non-zero real number. Let's fix a root $\beta = L_i - L_j$ which satisfy the assumption, a computation shows that:
\[ \cro{e_{\beta}}{\pi} = \lambda \left( 2 \sum_{i < k < j } e_{i,k} \wedge e_{k,j} - e_{i,j}\wedge ( h_i + h_{i+1} + \cdots +h_n) \right) \]
where $\{h_i=e_{i,i} - e_{i+1,i+1}\}_{1 \le i \le n}$  is the basis of the cartan subalgebra.
The coisotropic subalgebra thus obtained in $\gog$ is spanned by
\[ h_i + h_{i+1} + \cdots +h_n  ,\ e_{ij} ,\ \{ e_{kj} , e_{ik} \}_{i < k < j } \]
We will now restrict ourself without loose of generality in the case $i=1$ and $j=n$, and taking the chevalley generators, we obtain the coisotropic subalgebra $\goh$ spanned by:
\begin{align*} &h_1+h_{2}+ \cdots +h_n  ,\ e_1,\ \crc{e_1 , e_2},\ \crc{\crc{e_1,e_2},e_3}, \ \dots \ \crc{\crc{e_1,e_2}\dots,e_n} \\ 
                         &  e_n , \ \crc{e_n,e_{n-1}},\ \crc{\crc{e_n,e_{n-1}},e_{n-2}}, \ \dots  \ \crc{\crc{e_n,e_{n-1}}\dots,e_2} \end{align*}
\begin{equation*} \begin{pmatrix} a_0 & a_1 &\cdots & a_{n\text{-}1} & a_{n}   \\
                                   & 0     &  \cdots & 0        & b_{n\text{-}1}    \\
                                   &        & \ddots &  \vdots&\vdots  \\
                                   &       &              & 0        & b_{1} \\
                                  &        &              &           & \text{-}a_0 \end{pmatrix} \end{equation*}

We need to find a suitable candidate for the quantization. One way to proceed is to first take the subalgebra generated by
\begin{align*} & K_1K_2 \cdots K_n ,\ E_1,\ \crc{E_1,E_2},\ \crc{\crc{E_1,E_2},E_3},\ \dots \ \crc{\crc{E_1,E_2}\dots,E_n}  \\
& E_n, \ \crc{E_n,E_{n-1}},\ \crc{\crc{E_n,E_{n-1}},E_{n-2}}, \  \dots \ \crc{\crc{E_n,E_{n-1}}\dots,E_2} \end{align*}
but this subalgebra is not a coideal of $U_q(\gosl(n+1))$ therefore we need to change the generator a little. In fact we only need to change the power of the bracket to make it a coideal. We mean by that to take $\cro{E_1}{E_2}_q$. Let's proceed elements by elements. It is easy to see that
\begin{align*}  \Delta(K_1 \cdots K_n) =& \Delta(K_1)\Delta(K_2) \cdots \Delta(K_n) = K_1 \cdots K_n \otimes K_1 \cdots K_n \in A \otimes U_q(\gog) \\
 \Delta(E_1) =& E_1 \otimes K_1 + 1 \otimes E_1 \in B_h \otimes U_q(\gog) \end{align*}
therefore we do not need to change those. But for 
\[ \Delta( \cro{E_1}{E_2}) = \cro{E_1}{E_2} \otimes K_1K_2 + E_1 \otimes \cro{K_1}{E_2} + E_2 \otimes \cro{E_1}{K_2} + 1 \otimes \cro{E_1}{E_2} \]
there is one term that do not satisfy the condition here. We want this term to disapear. We want $\cro{E_1}{K_2} = 0$ but this is not true, but it is true for $\cro{E_1}{K_2}_q$.

\begin{Prop} 
For all $i \le n$ with $a_{k,k+1} = -1$ for all $k \in \{1,\dots ,i\}$, let's denote $X_i = \croq{\croq{E_1}{E_2}}{\dots , E_i}$  we have:
 \begin{align*} \Delta( X_i) =& 1 \otimes X_i  + E_1 \otimes \croq{\croq{K_1}{E_2}}{\dots , E_i} +  X_2 \otimes \croq{\croq{K_1K_2}{E_3}}{\dots , E_i} \\ & +
X_3 \otimes \croq{\croq{K_1K_2K_3}{E_4}}{\dots , E_i}  + \cdots + X_i \otimes K_1 \cdots K_i \end{align*}
\end{Prop}
The proof is done with an easy induction. By using this proposition, we have a suitable candidate for the quantization of $\goh$.  We note $B_h$ the subalgebra of $U_q(\gosl(n+1))$ generated by
\begin{align*} & K_1K_2 \cdots K_n ,\ E_1,\ \cro{E_1}{E_2}_q,\ \cro{\cro{E_1}{E_2}_q}{E_3}_q,\ \dots \cro{\cro{E_1}{E_2}_q}{\dots,E_n}_q  \\
 & E_n, \ \cro{E_n}{E_{n-1}}_q,\ \cro{\cro{E_n}{E_{n-1}}_q}{E_{n-2}}_q, \  \dots \cro{\cro{E_n}{E_{n-1}}_q}{\dots,E_2}_q \end{align*}
\begin{Prop} 
The subalgebra $B_h$ is a left coideal of the bialgebra $U_q(\goso(2n))$.
\end{Prop}
We construct $B_h$ to fullfill this condition. All we need to prove now is that $B_h$ is a flat deformation of $U(\goh)$ to prove that it is indeed a quantization of $\goh$. This proof is mainly computational.
\begin{Theo} 
$B_h$ is a quantization of $\goh$.
\end{Theo}
We will detail some of the computation as it is use in all the following computations.
Using corollary \ref{LemQ} we need to prove that for all generators $A_1$, $A_2$ we have $A_1A_2 - A_2A_1 = A' + hB$ where $A'$ is either a generator or 0 and B is in $B_h$. \newline
By computation, we will prove that this assertion is true. But we will only develop the non trivial computation. First, we will give some shortcuts:
\begin{Lem} \label{Astuce}
If $\cro{A}{B}_{q^a} = \cro{A}{C}_{q^b} = 0$ then $\cro{A}{\cro{B}{C}_{q^c}}_{q^{a+b}} = 0$ for all $a,b,c \in \mathbb{Z}$. \newline
If $\cro{A}{C}_{q^a} = \cro{B}{C}_{q^b} = 0$ then $\cro{\cro{A}{B}_{q^c}}{C}_{q^{a+b}} = 0$ for all $a,b,c \in \mathbb{Z}$.
\end{Lem}

let's denote $X_{j} = \cro{\cro{E_1}{E_2}_q}{\dots ,E_{j}}_q$. \newline 

$\bullet$ $A_1,A_2 \in ((1),(1))$: we can set $A_1 = X_j$ and $A_2 = X_{j+k}$, $j,k \in \mathbb{N}, \ \ j+k \le n.$.
We will use the following lemma:
\begin{Lem}
For all integers j,k such that $j + k \le n$ we have that $\cro{X_j}{X_{j+k}}_{q^{-1}} = 0$
\end{Lem}
\begin{proof}
Let's prove this lemma by induction on j.
It is easy to see that $\cro{E_1}{\cro{E_1}{E_2}_q}_{q^{-1}} = 0$ as it is the Quantum Serre relation.
By using the fact that $\cro{E_1}{E_3} = 0$ and the lemma \ref{Astuce}, we have that
\[ \cro{E_1}{\cro{\cro{E_1}{E_2}_q}{E_3}_q}_{q^{-1}} = 0.   \]
This can be extended to prove that $\cro{E_1}{X_j}_{q^{-1}} = 0$. \newline
If it hold for $j$, let's prove that it still hold for $j+1$.
\[ \cro{X_j}{X_{j+1}}_{q^{-1}} = \cro{\cro{X_{j-1}}{E_j}_q}{X_{j+1}}_{q^{-1}} \]
we have that $\cro{X_{j-1}}{X_{j+1}}_{q^{-1}} = 0$ by using the induction hypothesis. \newline
Furthermore, we can prove that  $\cro{E_j}{X_{j+1}} = 0$ .
\[ \begin{array}{rcl}
 \cro{E_j}{X_{j+1}} &=& \cro{E_j}{\cro{\cro{\cro{X_{j-2}}{E_{j-1}}_q}{E_j}_q}{E_{j+1}}_q} \\
                    &=& \cro{E_j}{\cro{X_{j-2}}{\cro{\cro{E_{j-1}}{E_j}_q}{E_{j+1}}_q}_q} \end{array} \]
We only need to prove that  $\cro{E_j}{\cro{\cro{E_{j-1}}{E_j}_q}{E_{j+1}}_q}=0$ because the rest is a consequence of lemma \ref{Astuce}.
\begin{Lem} \label{ijkj}
$E_j$ commutes with $\cro{\cro{E_i}{E_j}_q}{E_k}_q$ if we have $a_{ij} = a_{jk} = -1$ and $a_{ik}=0$ .
\end{Lem}
\begin{proof}
We will prove this lemma by computation.
\begin{eqnarray*}  \cro{\cro{\cro{E_i}{E_j}_q}{E_k}_q}{E_j} &=& E_iE_jE_kE_j - q E_kE_iE_jE_j - q E_jE_iE_kE_j + q^2 E_kE_jE_iE_j \\ & &- E_jE_iE_jE_k  + q E_jE_kE_iE_j + q E_jE_jE_iE_k - q^2 E_jE_kE_jE_i \end{eqnarray*}
and using the two Serre relations:
\[ E_iE_jE_j - (q + q^{-1})E_jE_iE_j + E_jE_jE_i = 0 \ \ (R_i) \]
and the other one by replacing $E_i$ with $E_k$.
\[ E_kE_jE_j - (q + q^{-1})E_jE_kE_j + E_jE_jE_k = 0 \ \ (R_k) \]
Using this relation, we can identify term by term to prove that
\[ \cro{\cro{\cro{E_i}{E_j}_q}{E_k}_q}{E_j}  = a*R_iE_k + b*E_kR_i + c* E_iR_k + d*R_kE_i\]
we find a linear system, that we solve:
\[ a = -\frac{1}{q+q^{-1}} , \  b = \frac{q^2}{q+q^{-1}}, \ c = \frac{1}{q+q^{-1}}, \ d = - \frac{q^2}{q+q^{-1}} .  \]
Therefore $\cro{\cro{\cro{E_i}{E_j}_q}{E_k}_q}{E_j} = 0$.
\end{proof}
Then, by using the lemmas \ref{ijkj} and \ref{Astuce}, we have that $\cro{X_j}{X_{j+1}}_{q^{-1}} = 0$. And we can extend this results to $\cro{X_j}{X_{j+k}}_{q^{-1}} = 0$.
\end{proof}

$\bullet$  $A_1,A_2 \in ((1),(b))$: we can set $A_1 = X_j$ and $A_2 = D_k = \cro{\cro{E_n}{E_{n-1}}_q}{\dots ,E_{k}}_q$. \newline
If $k \ge j + 2$ then $X_j$ and $D_k$ commute. \newline
If $k = j+1$, we can do an induction:
 \begin{align*}  \cro{X_{n-1}}{D_n}_q = &X_n \\
                        \cro{X_{n-2}}{D_{n-1}}_q =& \cro{X_{n-2}}{\cro{D_n}{E_{n-1}}_q}_q \\
                                              =& \cro{D_n}{\cro{X_{n-2}}{E_{n-1}}_q}_q \\
                                              =& - \cro{X_{n-1}}{D_n}_q + (1-q)(X_{n-1}D_n + D_nX_{n-1}) \\
                                              =& - X_n  + (1-q)(X_{n-1}D_n + D_nX_{n-1})                \end{align*} 
this is done by using the following lemmas.
\begin{Lem}
If A and B commute than $\cro{A}{\cro{B}{C}_q}_q = \cro{B}{\cro{A}{C}_q}_q$.
\end{Lem}
\begin{Lem} $ \cro{A}{B}_q = - \cro{B}{A}_q +(1-q)(AB + BA) $
\end{Lem}
We can reiterate this process for $X_{k}$ and $D_{k+1}$.
 \begin{align*}  \cro{X_{k}}{D_{k+1}}_{q} =& \cro{X_{k}}{\cro{D_{k+2}}{E_{k+1}}_q}_q \\
                                             =& \cro{D_{k+2}}{\cro{X_{k}}{E_{k+1}}_q}_q \\
                                             =& - \cro{X_{k+1}}{D_{k+2}}_q + (1-q)(X_{k+1}D_{k+2} + D_{k+2}X_{k+1}) \\
                                             =& (-1)^{n-k+1} \left( X_n - (1-q)\left(\sum_{j=k}^{n-2} (-1)^{j+1} (X_{j+1}D_{j+2} + D_{j+2}X_{j+1})\right)\right)     \end{align*}

for $k=j$ we have for $j=n$ that:
 \begin{align*} \cro{X_n}{E_n}_{q^{-1}} =& \cro{\cro{X_{n-2}}{\cro{E_{n-1}}{E_{n}}_q}_q}{E_n}_{q^{-1}} \\
                                              =& \cro{X_{n-2}}{\cro{\cro{E_{n-1}}{E_{n}}_q}{E_n}_{q^{-1}}}_q \ = \  0 \end{align*}
 Now for $j\leq n-1$
\begin{align*}   \cro{X_j}{D_j} =& \cro{\cro{X_{j-1}}{E_j}_q}{\cro{D_{j+1}}{E_j}_q} \\
                                     =& \cro{\cro{X_{j-2}}{\cro{E_{j-1}}{E_j}_q}_q}{\cro{D_{j+2}}{\cro{E_{j+1}}{E_j}_q}_q}\ = \ 0    \end{align*} 
to show that it is zero, it is enough to show $\cro{\cro{E_{j-1}}{E_j}_q}{\cro{E_{j+1}}{E_j}_q} = 0$.
\begin{Lem}
$\cro{E_i}{E_j}_q$ commutes with $\cro{E_k}{E_j}_q$ if we have $a_{ij} = a_{jk} = -1$ and $a_{ik}=0$.
\end{Lem}
\begin{proof}
We will prove this lemma by computation.
\begin{align*}  \cro{\cro{E_i}{E_j}_q}{\cro{E_k}{E_j}_q} =& E_iE_jE_kE_j - q E_iE_jE_jE_k - q E_jE_iE_kE_j + q^2 E_jE_iE_jE_k \\ & - E_kE_jE_iE_j  + q E_jE_kE_iE_j + q E_kE_jE_jE_i - q^2 E_jE_kE_jE_i \end{align*}
and using the two Serre relations:
\[ E_iE_jE_j - (q + q^{-1})E_jE_iE_j + E_jE_jE_i = 0 \ \ (R_i) \]
and the other one by replacing $E_i$ with $E_k$.
\[ E_kE_jE_j - (q + q^{-1})E_jE_kE_j + E_jE_jE_k = 0 \ \ (R_k) \]
Using this relation, we can identify term by term to prove that
\[ \cro{\cro{E_i}{E_j}_q}{\cro{E_k}{E_j}_q}  = a*R_iE_k + b*E_kR_i + c* E_iR_k + d*R_kE_i\]
We find a linear system, that we solve:
\[ a = -\frac{q^2}{q+q^{-1}} , \ b = \frac{1}{q+q^{-1}}, \ c = -\frac{1}{q+q^{-1}}, \ d = \frac{q^2}{q+q^{-1}} .  \]
Therefore $\cro{\cro{E_i}{E_j}_q}{\cro{E_k}{E_j}_q} = 0$
\end{proof}
for $k=j-1$ we have
\[ \cro{X_j}{D_{j-1}} = \cro{X_j}{\cro{D_j}{E_{j-1}}_q} = 0 \]
because $\cro{X_{j}}{D_j} = 0$ and $\cro{X_{j}}{E_{j-1}} = 0$ by using the same demonstration as in the lemma \ref{ijkj} .
This can be continued by induction, by decrementing k. \newline
Meaning that for $k < j-1$, we have
\[ \cro{X_j}{D_k} = \cro{X_j}{\cro{D_{k+1}}{E_{k}}_q} = 0 \]
because $\cro{X_{j}}{D_{k+1}} = 0$ and $\cro{X_{j}}{E_{k}} = 0$ by using the same demonstration as in the lemma \ref{ijkj}.  \newline

$\bullet$  $A_1,A_2 \in ((b),(b))$: it is the exact same proof as $\cro{(1)}{(1)}=0$ by reversing the indices. \newline
Finally, we have for all generators E of $B_h$ that there exist $l \in \mathbb{Z}$ such that 
 \[ \cro{\prod_{i=1}^{n} K_i}{E}=(1-q^l) \prod_{i=1}^{n} K_i E \]
Then by using proposition \ref{PropQ}, we can say that deformation is flat.

\section{$\goso(2n)$}

Following the construction, we construct coisotropic subalgebra $\goh$ in $\goso(2n)$ \newline
We consider $\gog$ with Cartan subalgebra given by the diagonal matrices. The roots will be given by $R = \{ \pm L_i \pm L_j \}_{i<j}$. it is easy to see that all the roots satisfy the assumption.
The root space of $\alpha = L_i - L_j$ is given by $e_{\alpha} = x_{i,j} = e_{i,j} - e_{n+j,n+i}$ and $f_\alpha = x{j,i}$, for $\alpha = L_i + L_j$ it is given by $e_{\alpha} = y_{i,j} = e_{i,n+j} - e_{j,n+i}$ and $f_{\alpha}=z_{j,i}=e_{n+j,i} - e_{n+i,j}$. We obtain the r-matrix
\[ \pi = \lambda \sum_{i < j}\left( x_ij \wedge x_ji + y_ij \wedge z_ij \right)   \ \ \  \lambda \in \mathbb{R}^*. \]

$\bullet$ We fix the root $\beta=L_i - L_j$. We then compute the bracket:
\[ \cro{x_{i,j}}{\pi} = \lambda \left( \sum_{i<k<j} x_{i,k} \wedge x_{k,j} +  x_{i,j} \wedge \cro{x_{i,j}}{x_{j,i}} \right) \]
The coisotropic subalgebra $\goh$ that we obtain, for a fixed i and j, in $\gog$ is generated by :
\[ \{ x_{ik} , x_{kj}\}_{i<k <j}, x_{ij}, \cro{x_{i,j}}{x_{j,i}} = h_i + h_{i+1} + \cdots + h_j \]
where $\{h_i = e_ii - e_{i+1,i+1} - e_{n+i,n+i} + e_{n+i+1,n+i+1}, h_n = e_{n,n} - e_{2n,2n}\}_{1\le i \le n-1}$ is the basis of the Cartan subalgebra which is in terms of chevalley generators:
\[h_i + h_{i+1} + \cdots + h_{j-1},\ e_i ,\ \cro{e_i}{e_{i+1}}, \ \cro{\cro{e_i}{e_{i+1}}}{e_{i+2}} , \ \dots \ , \ \cro{\cro{e_i}{e_{i+1}}}{\dots ,e_{j-1}} \]
\[ e_{j-1} ,\ \cro{e_{j-1}}{e_{j-2}}, \ \cro{\cro{e_{j-1}}{e_{j-2}}}{e_{j-3}} , \ \dots \ , \ \cro{\cro{e_{j-1}}{e_{j-2}}}{\dots,e_{i+1}} \]
 This example is the same as the case of $\gosl(n)$. \newline

$\bullet$ We now fix $\beta = L_i + L_j$.
The coisotropic subalgebra $\goh$ obtained in $\gog$ is generated by :
\[ \{ x_{i,k} , y_{k,j}\}_{i<k \neq j},\{ x_{j,k} , y_{k,i}\}_{j<k}, y_{i,j},\ \cro{x_{i,j}}{x_{j,i}} = h_i + h_{i+1} + \cdots + h_{j-1} \]
Where $\{h_i = e_ii - e_{i+1,i+1} - e_{n+i,n+i} + e_{n+i+1,n+i+1}, h_n = e_{n,n} - e_{2n,2n}\}_{1\le i \le n-1}$ is the basis of the Cartan subalgebra.
Without loosing any generality one can restrict the study to i=1. But we will distinct two case.  \newline
$\diamond$ If $j=n$ then the cosiotropic subalgebra $\goh$ will be generated in terms of chevalley generators by:
\[h_1 \cdots h_{n-1},\ e_1 ,\ \cro{e_1}{e_2}, \ \cro{\cro{e_1}{e_2}}{e_3} , \ \dots \ , \ \cro{\cro{e_1}{e_2}}{\dots ,e_{n-2}} \]
\[ e_n ,\ \cro{e_n}{e_{n-2}}, \ \cro{\cro{e_n}{e_{n-2}}}{e_{n-3}} , \ \dots \ , \ \cro{\cro{e_n}{e_{n-2}}}{\dots,e_{1}} \] 
\begin{equation*} \small \left( \begin{array}{ccccc|ccccc} 
                     a_0 & a_1 &\cdots & a_{n\text{-}2} & 0        & 0         &     0     & \cdots & 0         &\text{-}b_1 \\
                           & 0     &  \cdots & 0        &0        &  0         &          &            &     0       & \text{-}b_2 \\
                           &        & \ddots &  \vdots &\vdots & \vdots &           &           & \vdots       & \vdots\\
                           &       &              & 0         &   0      & 0         & 0 &  \cdots  &  0        & \text{-}b_{n\text{-}1} \\  
                          &        &              &           & \text{-}a_0    & b_1  & b_2     & \cdots & b_{n\text{-}1} & 0 \\
\hline 
                          &        &             &            &            & \text{-}a_0   &              &            &             &        \\
                          &        &             &            &            & \text{-}a_1   &   0         &            &             &        \\
                           &        &             &            &           & \vdots  &   \vdots &   \ddots &             &        \\
                            &        &             &            &          & \text{-}a_{n\text{-}2}&       0    &    \cdots&   0      &        \\
                             &        &             &            &            & 0        & 0  & \cdots  &    0    & a_0       \\
 \end{array} \right) \end{equation*}
This exemple is once again exactly the same as $\gosl(n+1)$. \newline
$\diamond$ If $j \neq n$ then it will be generated by:
\begin{align*} & \ \   h_1 + \cdots + h_{j-1}, e_1 ,\ \crc{e_1 ,e_2}, \ \crc{\crc{e_1 ,e_2} , e_3}  \dots  ,   \crc{\crc{e_1 ,e_2} \dots,e_{j-2}} \\
  & \ \ e_j ,\ \crc{e_j , e_{j+1}}, \ \crc{\crc{e_j , e_{j+1}} , e_{j+2}} \dots   ,  \crc{\crc{e_j , e_{j+1}}\dots ,e_{n-1}} \\
& \ \ \crc{e_j , t} , \ \crc{\crc{e_j , e_{j+1}} , t}, \
\crc{\crc{\crc{e_j , e_{j+1}} \dots ,e_{n-1}} , t} \\
 & \ \ {}_jx_n =  \crc{\crc{\crc{e_j , e_{j+1}}\dots ,e_{n-2}} , e_n},\  \crc{ {}_jx_n ,e_{n-1}}  \dots  , {}_jy_{j+1} = \crc{\crc{{}_jx_n , e_{n-1}} \dots ,e_{j+1}} \\
 & \ \ \cro{ {}_jx_n}{t}, \ \cro{\cro{ {}_jx_n}{e_{n-1}}}{t}, \dots \ , \ 
\cro{ {}_jy_{j+1}}{t} \\
&  \ \ \cro{ {}_jy_{j+1}}{\cro{e_j}{e_{j-1}}}, \ \dots \ , \cro{\cro{\cro{ {}_jy_{j+1}}{\cro{e_j}{e_{j-1}}}}{e_{j-2}}}{\dots ,e_1} \end{align*}
where $t = x_{j-1} = \cro{\cro{e_1}{e_2}}{\dots ,e_{j-1}}$ is not a generator.

\begin{equation*} \scriptsize  \left(\begin{array}{cccccccc|cccccccc}
a_0& a_1 &\cdots& a_{j\text{-}2}& 0      &  c_j    &\cdots & c_{n\text{-1}}    & 0    &0  &\cdots&  0   & \text{-}f_1 & e_j &\cdots& e_{n\text{-1}}   \\
     & 0     &\cdots& 0        &0      &  0     &\cdots &   0      &0  &     &       &     0      & \text{-}f_2 & 0 & \cdots  & 0    \\
    &        &\ddots&\vdots &\vdots&\vdots&   & \vdots& \vdots&    &     &  \vdots&\vdots&\vdots& & \vdots \\
  &        &          & 0         &   0      & 0       & \cdots &    0  &      0   & 0  &\cdots& 0   &\text{-}f_{j\text{-}1}&  0 &\cdots& 0\\ 
   &        &          &           & \text{-}a_0  & b_j     &  \cdots & b_{n\text{-1}}   & f_1 &f_2& \cdots& f_{j\text{-}1} & 0 & d_j &\cdots& d_{n\text{-}1} \\
   &        &          &           &          &  0       & \cdots & 0       & \text{-}e_j &    0  &\cdots   & 0          & \text{-}d_j& 0 & \cdots & 0 \\
   &        &          &           &          &          &  \ddots&\vdots&\vdots& \vdots &            &\vdots& \vdots&\vdots&  & \vdots \\
 &        &          &           &          &          &                & 0      & \text{-}e_{n\text{-1}}  & 0    &\cdots  & 0          & \text{-}d_{n\text{-1}}   & 0 &\cdots& 0 \\
\hline 
 &        &          &           &          &          &     &     &   \text{-}a_0 & & & & & & & \\
  &        &          &           &          &          &    &     &  \text{-}a_1 & 0 & & & & & & \\
 &        &          &           &          &          &   &     & \vdots&\vdots&\ddots&        &   &  &  &\\
 &        &          &           &          &          &    &     &\text{-}a_{j\text{-}2}& 0       & \cdots& 0    &     &  & & \\
  &        &          &           &          &          &   &     &  0       & 0       & \cdots & 0    & a_0& & &  \\
  &        &          &           &          &          &   &     &\text{-}c_{j}  & 0        & \cdots& 0      &\text{-}b_j& 0 & & \\
  &        &          &           &          &          &   &     & \vdots& \vdots&          &\vdots&\vdots&\vdots &\ddots & \\
  &        &          &           &          &          &   &     & \text{-}c_{n\text{-1}}  & 0& \cdots& 0 & \text{-}b_{n\text{-1}}  & 0& \cdots& 0 \end{array} \right) \end{equation*} 
We now need to choose a candidate for the quantization. following the method that we used for $\gosl(n+1)$, let's consider the following algebra generated by:
\begin{align*} (a) & \ \   K_1\cdots K_{j-1}, E_1 ,\ \crc{E_1 ,E_2}, \ \crc{\crc{E_1 ,E_2} , E_3}  \dots  ,   \crc{\crc{E_1 ,E_2} \dots,E_{j-2}} \\
 (b) & \ \ E_j ,\ \crc{E_j , E_{j+1}}, \ \crc{\crc{E_j , E_{j+1}} , E_{j+2}} \dots   ,  \crc{\crc{E_j , E_{j+1}}\dots ,E_{n-1}} \\
(c) & \ \ \crc{E_j , T} , \ \crc{\crc{E_j , E_{j+1}} , T}, \
\crc{\crc{\crc{E_j , E_{j+1}} \dots ,E_{n-1}} , T} \\
 (d) & \ \ {}_jX_n =  \crc{\crc{\crc{E_j , E_{j+1}}\dots ,E_{n-2}} , E_n},\  \crc{ {}_jX_n ,E_{n-1}}  \dots  , {}_jY_{j+1} = \crc{\crc{{}_jX_n , E_{n-1}} \dots ,E_{j+1}} \\
 (e) & \ \ \cro{ {}_jX_n}{T}, \ \cro{\cro{ {}_jX_n}{E_{n-1}}}{T}, \dots \ , \ 
\cro{ {}_jY_{j+1}}{T} \\
(f) &  \ \ \cro{ {}_jY_{j+1}}{\cro{E_j}{E_{j-1}}}, \ \dots \ , \cro{\cro{\cro{ {}_jY_{j+1}}{\cro{E_j}{E_{j-1}}}}{E_{j-2}}}{\dots ,E_1} \end{align*}

where $T = X_{j-1} = \cro{\cro{E_1}{E_2}}{\dots ,E_{j-1}}$ is not a generator. Each line corresponding to a set of generators. In order for the following computations, to be easier, let's take some notations:
let's denote: 
\begin{equation*} {}_jK_k = \prod_{i=j}^k K_i , \ \ {}_jX_k =\crc{\crc{E_j,E_{j+1}} \dots, E_k}, \ \ {}_jY_k = \cro{\cro{{}_jX_n}{E_{n-1}}}{\dots ,E_{k}} \ \ \  j < k \leq n \end{equation*}
\begin{equation*}
 {}_jY_k = \cro{\cro{\cro{ {}_jY_{j+1}}{\cro{E_j}{E_{j-1}}}}{E_{j-2}}}{\dots ,E_1}  \ \ \ k < j \end{equation*} 

We then want to change this subalgebra in order to make it into a left coideal. Therefore, the braquet in this notation may change depending on the case at hand.\newline 

$\bullet$ It is easy to see that the first two sets of generators are done by the same computation as in $\gosl(n+1)$. Therefore we now consider the two first set of generator with the q-bracket. \newline

$\bullet$ The third set is constituted by the bracket of the second set of generators with the element $T = \cro{\cro{E_1}{E_2}}{\dots ,E_{j-1}}$. The element $\Delta(T)$ can be developed as in $\gosl(n+1)$ meaning that we use the q-brackets. Therefore we now consider $T =\croq{\croq{E_1}{E_2}}{\dots ,E_{j-1}}$  . in $\Delta(T)$, the only term that fail as a coideal is $T \otimes  {}_1K_{j-1}$ Consequently, we only have to check that the bracket of the comultiplication of the second set of generators with this element is in $B \otimes U_q(\goso(2n))$. When computing $\Delta(\cro{{}_jX_k}{T})$, we see that only one term may pose a problem, $\cro{1 \otimes {}_jX_k}{T \otimes {}_1K_{j-1}}$.
But one can check that $\croq{\croq{\croq{E_j}{E_{j+1}}}{\dots,E_{j+k}}}{ {}_1K_{j-1}} = 0$. Therefore, as in the previous case, we only need to take the q-bracket. \newline

$\bullet$ For the fourth set, we can find by computation that we only need to take the q-bracket. 
\begin{Lem} In $U_q(\gog)$ for $k > j$, by taking ${}_jY_{k}= \croq{\croq{{}_jX_n}{E_{n-1}}}{\dots,E_{k}}$, we have that $\Delta({}_jY_k) \in B_h \otimes U_q(\goso(2n))$.
\end{Lem}
by a simple computation of all the term in ${}_jY_{k}$, we obtain by induction that:
\small{\begin{align*} \Delta({}_jY_k)
  =& 1 \otimes \crq{\crq{ {}_jX_n , E_{n-1}}\dots,E_k} + \ E_j \otimes \crq{\crq{\crq{\crq{\crq{K_j E_{j+1}} \dots,E_{n-2}}, E_n}, E_{n-1}} \dots,E_k} \\
  & +\  \cdots \ +  {}_jX_{k-1} \otimes \crq{\crq{\crq{\crq{\crq{{}_jK_{k-1} , E_k} \dots,E_{n-2}}, E_n} ,E_{n-1}} \dots,E_k} \\
  & + \ {}_jX_k \otimes\Big( \crq{\crq{\crq{\crq{\crq{{}_jK_{k} ,E_{k+1}}\dots,E_{n-2}},E_n},E_{n-1}}\dots,E_k} \\
  & \ \ \ \ \ \ \ \ \ \ \ \ \ \ \ \  + \crq{\crq{\crq{\crq{\crq{{}_jK_{k} ,E_k}\dots,E_{n-2}},E_n},E_{n-1}}\dots,E_{k+1}} \Big) \\
  & + \ {}_jX_{k+1} \otimes \Big( \crq{\crq{\crq{\crq{\crq{{}_jK_{k+1} ,E_{k+2}}\dots,E_{n-2}},E_n},E_{n-1}}\dots,E_k}  \\
  & \ \ \ \ \ \ \ \ \ \ \ \ \ \ \ \  + \crq{\crq{\crq{\crq{\crq{\crq{{}_jK_{k+1} ,E_{k+1}}\dots,E_{n-2}},E_n},E_{n-1}}\dots,E_{k+2}},E_k} \Big) \\
  & +\  \cdots \ +  {}_jX_{n-2} \otimes \Big( \crq{\crq{\crq{{}_jK_{n-2} ,E_{n}},E_{n-1}}\dots ,E_k}  \\
  & \ \ \ \ \ \ \ \ \ \ \ \ \ \ \ \ \ \ \ \ \ \ \ \ \ \   + \crq{\crq{\crq{\crq{\crq{{}_jK_{n-2},E_{n-2}},E_n},E_{n-1}},E_{n-3}}\dots ,E_k} \Big)  \\
  & +  \ {}_jX_{n-1} \otimes \crq{\crq{\crq{{}_jK_{n-1} ,E_{n}},E_{n-2}}\dots,E_k}   +  \ {}_jX_n \otimes \crq{\crq{{}_jK_{n-2} K_{n},E_{n-1}}\dots,E_k} \\
  & +  \ {}_jY_{n-1} \otimes \crq{\crq{{}_jK_{n-2} K_{n}K_{n-1},E_{n-2}}\dots,E_k}  + \ \cdots \ + \ {}_jY_k \otimes K_j \cdots K_{n-2}K_{n}K_{n-1} \cdots K_k \end{align*} }

$\bullet$ The fifth set is done exactly like the third one.  \newline

$\bullet$ Finally for the last set, we need to compute the different generators one by one.
One can find that
\begin{Lem} In $U_q(\gog)$ for $k < j$, by taking ${}_jY_{k}= \crq{\crq{ {}_jY_{j+1} , \crq{E_j , E_{j-1}}} \dots,E_{k}}$, we have  that $\Delta({}_jY_k) \in B_h \otimes U_q(\goso(2n))$.
\end{Lem}
The proof is done exactly like the preceding lemma.

The candidate $B_h$ that we choose, will be generated by:
\begin{align*} (a) & \ \   K_1\cdots K_{j-1}, E_1 ,\ \crq{E_1 ,E_2}, \ \crq{\crq{E_1 ,E_2} , E_3}  \dots  ,   \crq{\crq{E_1 ,E_2} \dots,E_{j-2}} \\
 (b) & \ \ E_j ,\ \crq{E_j , E_{j+1}}, \ \crq{\crq{E_j , E_{j+1}} , E_{j+2}} \dots   ,  \crq{\crq{E_j , E_{j+1}}\dots ,E_{n-1}} \\
(c) & \ \ \crq{E_j , T} , \ \crq{\crq{E_j , E_{j+1}} , T}, \
\crq{\crq{\crq{E_j , E_{j+1}} \dots ,E_{n-1}} , T} \\
 (d) & \ \ {}_jX_n =  \crq{\crq{\crq{E_j , E_{j+1}}\dots ,E_{n-2}} , E_n},\  \crq{ {}_jX_n ,E_{n-1}}  \dots  , {}_jY_{j+1} = \crq{\crq{{}_jX_n , E_{n-1}} \dots ,E_{j+1}} \\
 (e) & \ \ \crq{ {}_jX_n , T}, \ \crq{\crq{ {}_jX_n ,E_{n-1}} , T}, \dots \ , \ 
\crq{ {}_jY_{j+1} ,T} \\
(f) &  \ \ \crq{ {}_jY_{j+1} , \crq{E_j , E_{j-1}}}, \ \dots \ , \crq{\crq{\crq{ {}_jY_{j+1} , \crq{E_j ,E_{j-1}}} ,E_{j-2}} \dots ,E_1} \end{align*}
where $T = X_{j-1} = \cro{\cro{E_1}{E_2}}{\dots ,E_{j-1}}$ is not a generator. Each line corresponding to a set of generators.
\begin{Prop} 
The subalgebra $B_h$ is a left coideal of the bialgebra $U_q(\goso(2n))$.
\end{Prop}
We need to check that this deformation is flat to prove that it is indeed a quantization.
\begin{Theo} 
$B_h$ is a quantization of $\goh$.
\end{Theo}
\begin{proof}: 
By computation, we will prove that the deformation is flat.\newline

  $\bullet$ For $A_1,A_2 \in ((a),(a))$, the demonstration is the same as $\gosl(n+1)$.  \newline

  $\bullet$ For $A_1,A_2 \in ((a),(b))$ or $((a),(d))$, we have that $\crc{A_1,A_2}=0$, because they commute.  \newline
  
  $\bullet$ For $A_1,A_2 \in ((a),(c))$or $ ((a),(e))$, it is equivalent to $\cro{A_1}{T}$ where $T = \cro{\cro{E_1}{E_2}}{\dots,E_{j-1}}$. By using the same argument as the proof in $\gosl(n+1)$, we can prove that $\cro{A_1}{T} = 0$. Also, we have that $A_1$ commutes with the elements in (b) and (d), therefore it commutes with $A_2$. \newline 

  $\bullet$ For $A_1,A_2 \in ((a),(f))$, it is done in the following way:\newline
set $A_1=X_k=\cro{\cro{E_1}{E_2}}{\dots,E_k}$ for $1\le k \le j-2$ and $A_2={}_jY_l=\cro{\cro{ {}_jY_{j+1}}{\cro{E_j}{E_{j-1}}}}{\dots,E_{l}}$ for $1 \le l \le j-1$. We have to examine $\cro{X_k}{ {}_jY_l}$. \newline 
$\diamond$  If $k < l-1$, then it is easy to see that $X_k$ and ${}_jY_l$ commute. \newline
 $\diamond$ If $k = l-1$, we then need to consider each case. for k=1, we have:
\[ \cro{ {}_jY_{2}}{X_1}_q = {}_jY_1 \]
for k=2, we will use the following propertie, for every A, B and $a \in \mathbb{Z}$, we have that $\cro{A}{B}_q = - \cro{B}{A}_{q} + (1-q)(BA-AB)$.  Then by a simple computation, we have
\begin{align*}  \crq{  {}_jY_3 ,X_2}=& \cro{ {}{}_jY_{3}}{\cro{X_1}{E_2}_q}_{q} \\
                                            =& \cro{X_1}{\cro{Y_{3}}{E_2}_q}_q \\
                                            =& -\cro{Y_2}{X_1}_q  + (1-q)(Y_2X_1-X_1Y_2) \\
                                            =& -\cro{Y_2}{X_1}_q  + (1-q)(Y_2X_1-X_1Y_2) \end{align*}
And by successive iteration, we can find that:
\begin{align*} \cro{Y_{k+1}}{X_k}_{q} =& \cro{Y_{k+1}}{\cro{X_{k-1}}{E_k}_q}_{q} \\
                                               =& \cro{X_{k-1}}{\cro{Y_{k+1}}{E_k}_q}_q \\\
                                               =& - \cro{Y_k}{X_{k-1}}_{q}  + (1-q)( Y_kX_{k-1} - X_{k-1}Y_k) \\
                                               =& (-1)^{k-1} Y_1 + (1-q) \left(\sum_{i=1}^{k-1} (-1)^{j-1-i} ( Y_{i+1}X_{i} - X_iY_{i+1}) \right)             \end{align*}
$\diamond$ If $j-2 \ge k \ge l$ then we have to consider:
\begin{align*}  \cro{X_k}{ {}_jY_l} =& \cro{X_k}{ \cro{ {}_jY_{j+1}}{\cro{\cro{E_j}{E_{j-1}}_q}{\dots,E_l}_q}_q} \\
                                           =& \cro{ {}_jY_{j+1}}{\cro{X_k}{\cro{\cro{E_j}{E_{j-1}}_q}{\dots,E_l}_q}}_q  \end{align*}
We can verify that $\cro{X_k}{\cro{\cro{E_j}{E_{j-1}}_q}{\dots,E_l}_q}= 0$ by using the fact that :
 \[ X_k = \cro{\cro{\cro{X_{l-2}}{\cro{E_{l-1}}{E_l}_q}_q}{E_{l+1}}_q}{\dots,E_k}_q \] 
We have that $X_{l-2}$ and ${}_jX_{l} = \cro{\cro{E_j}{E_{j-1}}_q}{\dots,E_l}_q$ commute. It is the same for $\cro{E_{l-1}}{E_l}_q$, $E_{l+1}$ \dots,$E_k$. Therefore $X_k$ and $\cro{\cro{E_j}{E_{j-1}}_q}{\dots,E_l}_q$ commute. \newline 

  $\bullet$ For $A_1,A_2 \in ((b),(b))$, the demonstration is the same as $\gosl(n+1)$. \newline

  $\bullet$ For $A_1,A_2 \in ((b),(c))$, is equivalent to $\cro{(b)}{\cro{(b)}{T}}$. We have in fact to compute $\cro{ {}_jX_k}{\cro{ {}_jX_l}{T}}$ and this is done just like in $\gosl(n+1)$. We find that if k < l then we just have to use the lemma \ref{ijkj}. If k = l, then we have that  $\cro{ {}_jX_k}{\cro{ {}_jX_l}{T}}_{q^{-1}}= 0$ by using the same proof as in $\gosl(n+1)$. And if k > l, then $\cro{ {}_jX_k}{\cro{ {}_jX_l}{T}}= 0$ which is also done like $\gosl(n+1)$. \newline 
  
  $\bullet$ For $A_1,A_2 \in ((b),(d))$, we need to examine $\cro{ {}_jX_k}{Y_l}$, for $j\le k \le n-1$ and $j+1 \le l \le n$.  \newline
$\diamond$ If $k < l-1$, then we have that:
\begin{align*} \cro{ {}_jX_k}{Y_l}_{q^{-1}} =& \cro{ {}_jX_k}{\cro{\cro{ {}_jX_n}{E_{n-1}}_q}{\dots,E_l}_q}_{q^{-1}} \\
                                        =& \cro{\cro{\cro{ {}_jX_k}{ {}_jX_n}_{q^{-1}}}{E_{n-1}}_q}{\dots,E_l}_q \end{align*} 
We can verify that $\cro{ {}_jX_k}{ {}_jX_n}_{q^{-1}}=0$ for $k \le n-2$ just by using the same proof as in $\gosl(n+1)$.
Therefore $\cro{ {}_jX_k}{Y_l}_{q^{-1}} = 0$ for $k < l-1$. \newline
$\diamond$ If $k=l-1$, then for $k=n-1$ and $l=n$, we have:
\[ \cro{ {}_jX_{n-1}}{ {}_jX_n} = \cro{\cro{ {}_jX_{n-2}}{E_{n-1}}_q}{\cro{ {}_jX_{n-2}}{E_{n}}_q} \]
let's set $A={}_jX_{n-2}$, $B=E_{n-1}$ and $C = E_{n}$.
We are in the same settings as the lemma in $\gosl(n+1)$.
Therefore $\cro{ {}_jX_{n-1}}{ {}_jX_n} = \cro{\cro{A}{B}_q}{\crq{A , C}} = 0$. \newline
 For k=n-2 and l=n-1, we have:
\begin{align*}  \cro{ {}_jX_{n-2}}{ {}_jY_{n-1}} =& \cro{ {}_jX_{n-2}}{\cro{ {}_jX_n}{E_{n-1}}_q} \\
                                                                        =& q^{-1} \cro{ {}_jX_n}{\cro{ {}_jX_{n-2}}{E_{n-1}}_q}_{q^2} \\
                                                   =& - q^{-1} \cro{ {}_jX_{n-1}}{ {}_jX_n} + q^{-1}(1-q^2)({}_jX_{n-1} \ {}_jX_n) \\
                                                   =&  q^{-1}(1-q^2)({}_jX_{n-1} \ {}_jX_n) \end{align*} 
And by successive iteration we can find that:
\begin{align*}  \crc{ {}_jX_{k} ,  {}_jY_{k+1}}  = & \crc{ {}_jX_{k} , \crq{ {}_jY_{k+2} , E_{k+1}}}    \\
                                                = & q^{-1} \cro{ {}_jY_{k+2}}{\cro{ {}_jX_{k}}{E_{k+1}}_q}_{q^2} \\
                                                 = & - q^{-1} \cro{ {}_jX_{k+1}}{ {}_jY_{k+2}} + q^{-1}(1-q^2)({}_jX_{k+1} {}_jY_{k+2}) \\ 
                                                = &  (1-q^2) \left( \sum_{i=1}^{n-k-1} (-q)^{-i}({}_jX_{k+i}  {}_jY_{k+1+i}) \right)                 \end{align*} 
$\diamond$ If $k \ge l$, then we need to consider:
 \begin{align*}  \cro{ {}_jX_{k}}{ {}_jY_{l}}_{q^{-1}} =& \cro{ {}_jX_{k}}{\cro{\cro{ {}_jY_{k}}{E_{k-1}}}{\dots,E_l}} \\
                                             =&  \cro{\cro{\cro{ {}_jX_{k}}{ {}_jY_{k}}_{q^{-1}}}{E_{k-1}}}{\dots,E_l} \end{align*} 
We will consider  $\cro{ {}_jX_{k}}{ {}_jY_{k}}_{q^{-1}}$. For k=n-1, we have:
\[ \cro{ {}_jX_{n-1}}{ {}_jY_{n-1}}_{q^{-1}} = \cro{ {}_jX_{n-1}}{\cro{ {}_jX_{n}}{E_{n-1}}}_{q^{-1}} \]
But we have that $\cro{ {}_jX_{n-1}}{ {}_jX_{n}}=0$ and $\cro{ {}_jX_{n-1}}{E_{n-1}}_{q^{-1}} = 0$ therefore:
$\cro{ {}_jX_{n-1}}{ {}_jY_{n-1}}_{q^{-1}} =0$.
for $k \le n-2$. We have:
\begin{align*}  \cro{ {}_jX_{k}}{ {}_jY_{k}}_{q^{-1}} =& \cro{ {}_jX_{k}}{\cro{ {}_jY_{k+2}}{\cro{E_{k+1}}{E_k}}}_{q^{-1}} \\
                                        =& q^{-1} \cro{ {}_jY_{k+2}}{\cro{ {}_jX_{k}}{\cro{E_{k+1}}{E_k}}}_{q^2} \end{align*} 
but we have that  $\cro{ {}_jX_{k}}{\cro{E_{k+1}}{E_k}}=0$ by using the proof of $\gosl(n+1)$.
Therefore $\cro{ {}_jX_{k}}{ {}_jY_{l}}_{q^{-1}} = 0$ for $k \ge l$. \newline

  $\bullet$ For $A_1,A_2 \in ((b),(e))$, it is equivalent to $\cro{(b)}{\cro{(d)}{T}}$. Is done exactly the same way as the previous one by considering the fact that $(e) = \cro{ {}_jY_l}{T} = \cro{\cro{\cro{ {}_jX_{n}}{T}}{E_{n-1}}}{\dots,E_l}$. \newline

  $\bullet$ For $A_1,A_2 \in ((b),(f))$, this is proved by using the fact that we only need to consider this calculus for the element ${}_jY_{j-1}$ for $(f)$ because all the other calculus are done trivially using this element. \newline
One can see that we have for $j \leq k \leq n-2$:
\[ \cro{ {}_jX_k}{ {}_jY_{j-1}}_{q^{-1}} = \cro{ {}_jX_k}{\cro{ {}_jY_{k+2}}{ _{k+1}D_{j-1}}_q}_{q^{-1}} \]
where  $_{k+1}D_{j-1} = \cro{E_{k+1}}{\cro{E_k}{\dots \cro{E_j}{E_{j-1}}}}$. And by using the same method as in $\gosl(n+1)$, we have that ${}_jX_k$ commutes with $_{k+1}D_{j-1}$. And by using the previous calculus, we have that ${}_jX_k$ $q^{-1}-$ commutes with ${}_jY_{k+2}$ and therefore $\cro{ {}_jX_k}{ {}_jY_{j-1}}_{q^{-1}} = 0$. \newline
One last computation for k = n-1.
\[ \cro{ {}_jX_{n-1}}{ {}_jY_{j-1}}_{q^{-1}} = \cro{ {}_jX_{n-1}}{\cro{ {}_jY_n}{ _{n-1}D_{j-1}}_q}_{q^{-1}} \]
By using the same method as in $\gosl(n+1)$, we have that ${}_jX_{n-1}$ $q^{-1}$-commutes with $_{n-1}D_{j-1}$. And by using the previous calculus, we have that ${}_jX_{n-1}$ commutes with ${}_jY_{n}$ and therefore $\cro{ {}_jX_{n-1}}{ {}_jY_{j-1}}_{q^{-1}} = 0$. \newline

$\bullet$ For the remaining, case, it is either done like previously or by using some simples arguments.
By using the proposition \ref{LemQ}, we finish our proof.
\end{proof}

\section{$\gosp(2n)$}

Following the construction, we construct coisotropic subalgebra in $\goh \subset \gosp(2n)$. \newline
We consider $\gog$ with Cartan subalgebra given by the diagonal matrices. The roots will be given by $R= \{ \pm L_i \pm L_j \}$. The roots satisfying the assumption are of the form $\pm 2 L_i$.
The root space of $\alpha = L_i - L_j$ is given by $e_{\alpha} = x_{i,j} = e_{i,j} - e_{n+j,n+i}$ and $f_\alpha = x{j,i}$, for $\alpha = L_i + L_j$ it is given by $e_{\alpha} = y_{i,j} = e_{i,n+j} + e_{j,n+i}$ and $f_{\alpha}=z_{i,j}=e_{n+i,j} + e_{n+i,j}$ and finally for $\alpha = 2 L_i$ it is given by $e_{\alpha} = u_i = E_{i,n+i}$ and $f_{\alpha} = v_i = e_{n+i,i}$. We obtain the r-matrix
\[ \pi = \lambda \left( \frac12 \sum_{i < j}\left( x_ij \wedge x_ji + y_ij \wedge z_ij \right) + \sum_i u_i \wedge v_i \right) \]
where $\lambda \in \mathbb{R}^*$. We fix the root $\beta=2L_i$. We then compute the bracket:
\[ \cro{u_i}{\pi} = \lambda \left( \sum_{i<j} y_{i,j} \wedge x_{i,j} +  u_i \wedge h_i \right) \]
Where $\{h_i = e_ii - e_{i+1,i+1} - e_{n+i,n+i} + e_{n+i+1,n+i+1}, h_n = e_{n,n} - e_{2n,2n}\}_{1\le i \le n-1}$ is the basis of the Cartan subalgebra.
The coisotropic subalgebra $\goh$ that we obtain, in $\gog$ is generated by :
\[ \{ y_{i,k} , x_{i,k}\}_{i<k},\ u_i,\ h_i + h_{i+1} \dots + h_n \]
Without loose of generality, one can restrict the study to i=1, the other case being equivalent to the first one in lower dimension. Then, the coisotropic subalgebra $\goh$ that we hence obtain, is generated by:
\begin{align*} 
(a) & \ \  h_1 +\cdots + h_{n}, \ e_1,\ \cro{e_1}{e_2},\ \cro{\cro{e_1}{e_2}}{e_3}, \ \dots \ , \cro{\cro{e_1}{e_2}}{\dots,e_{n-1}} \\
(b)  &\ \  x_n \ = \ \cro{\cro{e_1}{e_2}}{\dots,e_{n}}, \ \cro{x_n}{e_{n-1}},\ \cro{\cro{x_n}{e_{n-1}}}{e_{n-2}}, \ \dots \ , \cro{\cro{x_n}{e_{n-1}}}{\dots,e_{1}} 
\end{align*}
\begin{equation*} \left( \begin{array}{cccc|cccc} 
                     a_0 & a_1 &\cdots & a_{n\text{-}1}     & b_1        &     b_2     & \cdots & b_n  \\
                           & 0     &  \cdots & 0           &  b_2         &     0     & \cdots    &     0       \\
                           &        & \ddots &  \vdots  & \vdots &     \vdots  &           & \vdots    \\
                           &       &              & 0         &  b_n         & 0 &  \cdots  &  0         \\  
                        
\hline 
                          &        &             &      & \text{-}a_0   &              &            &                     \\
                          &        &             &       & \text{-}a_1   &   0         &            &                    \\
                           &        &             &        & \vdots  &   \vdots &   \ddots &                     \\
                            &        &             &         & \text{-}a_{n\text{-}1}&       0    &    \cdots&   0             \\
                           
 \end{array} \right) \end{equation*}

The candidate $B_h$ that we choose to be the quantization of $\goh$ in $U_q(\gosp(2n))$, will be generated by:
\begin{align*}
(a) & \ \   K_1\cdots K_{n}, \ E_1,\ \cro{E_1}{E_2}_q,\ \cro{\cro{E_1}{E_2}_q}{E_3}_q, \ \dots \ , \cro{\cro{E_1}{E_2}_q}{\dots,E_{n-1}}_q \\
 (b) & \ \ X \  = \ \cro{\cro{E_1}{E_2}_q}{\dots,E_{n}}_{q^2}, \ \cro{X}{E_{n-1}}_q,\ \cro{\cro{X}{E_{n-1}}_q}{E_{n-2}}_q, \ \dots \ , \cro{\cro{X}{E_{n-1}}_q}{\dots,E_{1}}_q 
\end{align*}
We now have to check if $\Delta(B_h) \subset B_h \otimes U_q(\gosp(2n))$.
\begin{Prop} 
The subalgebra $B_h$ is a left coideal in $U_q(\gosp(2n))$.
\end{Prop}
\begin{proof} 

$\bullet$ It is easy to see that the first set of generators will satisfy this property by using the same demonstration as in $\gosl(n+1)$. We need to check the property with the second set of generators. One can check that for $\cro{\cro{E_1}{E_2}}{\dots,E_{n}}_{q^2}$, it is almost the same as in $\gosl(n+1)$.
we just need to see that:
\[ \cro{E_{n-1}}{E_n}_{q^2} = 1 \otimes \cro{E_{n-1}}{E_n}_{q^2} + E_{n-1} \otimes \cro{K_{n-1}}{E_n}_{q^2} + \cro{E_{n-1}}{E_n}_{q^2} \otimes K_nK_{n-1} \]
We do not see $E_n$ because we have that $\cro{E_{n-1}}{K_n}_{q^2} = 0$. \newline

$\bullet$ For the remaining ones, we will have to do an other induction.

\begin{Lem} In $U_q(\gosp(2n))$, we have to take $Y_{k} = \crq{\crq{X_n,E_{n-1}}\dots,E_{k}}$, for $k \in \{1,\dots,n-1\}$, to have $\Delta(Y_{k}) \in B_h \otimes  U_q(\gosp(2n))$
\end{Lem}
the proof is done by computation. One can find that:
\begin{align*} \Delta(Y_{k}) =& 1 \otimes \crq{\crq{\crqq{\crq{E_1,E_2}\dots,E_{n}},E_{n-1}}\dots,E_{k}}  +  E_1 \otimes \crq{\crq{\crqq{\crq{K_1,E_2}\dots,E_{n}},E_{n-1}}\dots,E_{k}} \\
 & + \ \cdots \ + \ X_{k-1} \otimes \crq{\crq{\crqq{\crq{ {}_1K_{k-1} ,E_{k}}\dots,E_{n}},E_{n-1}}\dots,E_{k}} \\
 & +\  X_{k} \otimes \Big( \crq{\crq{\crqq{\crq{ {}_1K_{k} ,E_{{k}+1}}\dots,E_{n}},E_{n-1}}\dots,E_{k}} \\
 & \ \ \  \ \ \ \ \ \ \ \ \  + \crq{\crq{\crqq{\crq{ {}_1K_{n} ,E_{k}}\dots,E_{n}},E_{n-1}}\dots,E_{{k}+1}} \Big)  \\
 & + \ \cdots \  + \ X_{n-1} \otimes \Big( \crq{\crq{\crqq{{}_1K_{n-1},E_n},E_{n-1}}\dots,E_{k}} \\
 &\ \ \ \ \ \ \ \ \ \ \ \ \ \ \ \  \ \ \ \ \ \ \ \  + \crq{\crq{\crqq{\crq{{}_1K_{n-1},E_{n-1}},E_{n}},E_{n-2}}\dots,E_{{k}}} \Big) \\
 & + \ X_{n} \otimes \crq{\crq{{}_1K_{n},E_{n-1}}\dots,E_{k}}  +  Y_{n-1} \otimes  \crq{\crq{{}_1K_{n}K_{n-1},E_{n-2}}\dots,E_{k}} \\
 & + \ \cdots \  + \ Y_{k} \otimes {}_1K_{n} \ {}_{n-1}K_{k}
\end{align*} 
\end{proof}
Then again as what we did in the last part, we need to check if this quantization is flat. And we will follow the exact same demonstration.
\begin{Theo} 
$B_h$ is a quantization of $\goh$.
\end{Theo}
 \begin{proof}
 Using the lemma \ref{LemQ}, we need to prove that for all generators $A_1$, $A_2$. By computation, we will prove that this assertion is true. \newline
 
$\bullet$ $A_1,A_2 \in (a),(a)$, it is done the same way as in the previous example $\gosl(n+1)$. \newline

$\bullet$ $A_1,A_2 \in (a),(b)$. We can set $A_1 = X_j$ and $A_2 = Y_{k}$ with $1 \leq k,j \leq n$. \newline 
$\diamond$ If $k \geq j+2$. We have:
\[ \crc{X_j , Y_k}_{q^{-1}} = \crc {X_j , \crq{\crq{\crqq{X_{n-1} , E_n} , E_{n-1}} , \dots,E_{k}}}_{q^{-1}}  = 0\]
By using the  fact that $\crc{X_j , X_{n-1}}_{q^{-1}} = 0$(given by the previous example) and the fact $X_j$ commutes with $E_n,E_{n-1}\dots,E_k$ for $k \ge j+2$. \newline
 $\diamond$ If $k = j+1$ . First for j=1, we have:
\[ \crq{Y_{2} , X_1} = Y_1 \]
 For j=2, we will use the following lemma:
\begin{Lem} If $\crc{B , A}_{q^{-1}}=0$, then  $\crqq{A , \crq{B , C}} = q \crc{B , \crq{C , A}}$. \newline
For every A, B and $a \in \mathbb{Z}$, we have that $\crc{A, B} = - \crc{B , A}_{a} + (1-q^a)(AB)$.
\end{Lem} \begin{align*} \crqq{Y_{3} , X_2}  =& \crqq{Y_{3} , \crq{X_1 , E_2}} \\
                                            =& q  \crc{X_1 , \crq{Y_{3} , E_2}} \\
                                            =& -q \left( \crq{Y_2, X_1}  - (1-q)(X_1Y_2) \right) \\
                                            =& -q Y_1 + (1-q)q(X_1Y_2) \end{align*}
And by successive iteration we can find that:
\begin{align*}  \crqq{Y_{j+1} , X_j}=& \crqq{Y_{j+1} , \crq{X_{j-1} , {E_j}}} \\
                                               =& q  \crc{X_{j-1} , \crq{Y_{j+1} , E_j}} \\
                                               =& -q \left( \crqq{Y_j , X_{j-1}}  - (1-q^2)(X_{j-1}Y_j) \right) \\
                                               =& (-q)^{j-1} \left( Y_1  +(1-q) X_1Y_2 \right) - \left((1-q^2) \sum_{i=2}^{j-1} (-q)^{j-i} X_iY_{i+1} \right) \end{align*}
$\diamond$ If $k = j$ , we have for $j=n$, $\cro{X_n}{X_n} = 0$. If $j = n-1$, we have that:
\[ \crc{X_{n-1},Y_{n-1}}_{q^{-1}} = \crc{\crc{X_{n-2},E_{n-1}}_q,\crc{\crc{\crc{X_{n-2},E_{n-1}}_q,E_n}_{q^2},E_{n-1}}_q}_{q^{-1}} \]
By setting $A=X_{n-2}$, $B=E_{n-1}$ and $C=E_n$ and by using the following relation $\crc{A,\crc{A,B}_q}_{q^{-1}}$ , $\crc{B,\crc{B,A}_q}_{q^{-1}}$, $\crc{B,\crc{B,\crc{B,C}_{q^2}} }_{q^{-2}}$. We can prove that $\crc{X_{n-1},Y_{n-1}}_{q^{-1}} = 0$. For $j \le n-2$, we have that:
\begin{align*}  \crc{X_{j},Y_{j}}_{q^{-1}} =& \crc{X_j,\crc{Y_{j+2},\crc{E_{j+1},E_{j}}_q}_q}_{q^{-1}} \\
                                         =& q^{-1} \crc{Y_{j+2},\crc{X_j,\crc{E_{j+1},E_j}_q}}_{q^2} \end{align*}
Which is zero by using the same proof as in $\gosl(n+1)$. i.e. $\crc{X_j,\crc{E_{j+1},E_j}_q} = 0$.                          \newline
 $\diamond$ Finally if $n \neq k > j$, we have to consider:
\begin{align*} \crc{X_k,Y_j}_{q^{-1}} =& \crc{X_k,\crc{Y_k,\crc{E_k-1,\dots,E_j}_q}_q}_{q^{-1}} \\
                                   =& \crc{\crc{X_k,Y_k}_{q^{-1}},\crc{E_k-1,\dots,E_j}_q}_q \ = 0  \end{align*}
Because $\crc{E_k-1,\dots,E_j}_q$ commutes with $X_k$ by using the same proof as in $\gosl(n+1)$ and  $\crc{X_k,Y_k}_{q^{-1}} = 0$. For k=n, we have to consider the special case of
\[  \crc{X_n,Y_{n-1}}_{q^{-1}} = \crc{ \crqq{X_{n-1},E_n} ,Y_{n-1}}_{q^{-1}} = 0  \]
which is solved by using the fact that $\crc{X_{n-1},Y_{n-1}}_{q^{-1}} = \crc{E_n ,Y_{n-1}} = 0$. Then we extend this to $Y_j$ by using the fact $E_j$ commutes with $X_n$ for $j \leq n-2$. \newline

$\bullet$  $A_1,A_2 \in (b),(b)$, we need to compute $\crc{Y_k,Y_{k-1}}_{q^{-1}}$ with $k \ge j$. We can see that:
\begin{align*} \crc{Y_k,Y_{k-1}}_{q^{-1}} =&  \crc{Y_k,\crc{Y_k+1,\crc{E_k,E_{k-1}}_q}_q}_{q^{-1}} \\
                                             =& \crc{\crc{Y_k,Y_{k+1}}_{q^{-1}},\crc{E_k,E_{k-1}}_q}_q \end{align*}
Because $\crc{E_k,E_{k-1}}_q$ commutes with $X_n$ and $_{n-1}D_k$ therefore it commutes with $Y_k$. By using this, we only need to consider the final case $\crc{Y_{n-1},Y_{n-2}}_{q^{-1}}$ which is zero by using the same relation as the previous case. \newline

$\bullet$ For all E generators in $B_h$, there exists $l \in \mathbb{N}$ such that:
 \[ \cro{K_1\cdots K_j}{a} = (1 - q^l) K_1\cdots K_j E . \]
By using the proposition \ref{LemQ}, we finish our proof.  \end{proof}

\section{$\goso(2n+1)$}

Following the same construction, we construct a coisotropic subagebra $\goh$ in $\gog =\goso(2n+1)$ \newline
We consider $\gog$ with Cartan subalgebra given by the diagonal matrices. The roots are R=$\{\pm L_i \pm L_j\}_{i < j} \cup \{\pm L_i\}$. 
The roots that satisfy the assumption are those of the form $\{\pm L_i \pm L_j\}_{i < j}$. \newline 
The root space of $\alpha = L_i - L_j$ is spanned by $e_{\alpha} = x_{i,j} = e_{i,j} - e_{n+j,n+i}$ and $f(\alpha) =x_{j,i}$ . For $\alpha = L_i + L_j$ it is given by $e_{\alpha} = y_{i,j} = e_{i,n+j} - e_{j,n+i}$ and $f_{\alpha}=z_{i,j}=y_{i,j}^t$. And finally for $\alpha = L_i$ it is given by $e_{\alpha} = u_i = e_{i,2n+1} - e_{2n+1,n+i}$ and $f(\alpha)=v_i = u_i^t$. We obtain the r-matrix
\[ \pi = \lambda \left( \frac12 \sum_{i < j}\left( x_ij \wedge x_ji + y_ij \wedge z_ij \right) + \sum_i u_i \wedge v_i \right) \]
where $\lambda \in \mathbb{R}^*$. \newline

$\beta=L_i-L_j$. The coisotropic subalgebra $\goh$ that we obtain, for a fixed i and j, in $\gog$ is generated by :
\[ \{ x_{ik} , x_{kj}\}_{i<k <j}, x_{ij}, \cro{x_{i,j}}{x_{j,i}} = h_i + h_{i+1} + \cdots + h_j \]
Where $\{h_i = e_ii - e_{i+1,i+1} - e_{n+i,n+i} + e_{n+i+1,n+i+1}, h_n = e_{n,n} - e_{2n,2n}\}_{1\le i \le n-1}$ is the basis of the Cartan subalgebra which is in terms of chevalley generators:
\[h_i + h_{i+1} + \cdots + h_{j-1},\ e_i ,\ \cro{e_i}{e_{i+1}}, \ \cro{\cro{e_i}{e_{i+1}}}{e_{i+2}} , \ \dots \ , \ \cro{\cro{e_i}{e_{i+1}}}{\dots ,e_{j-1}} \]
\[ e_{j-1} ,\ \cro{e_{j-1}}{e_{j-2}}, \ \cro{\cro{e_{j-1}}{e_{j-2}}}{e_{j-3}} , \ \dots \ , \ \cro{\cro{e_{j-1}}{e_{j-2}}}{\dots,e_{i+1}} \]
 This example is the same as the case of $\gosl(n)$. \newline

$\beta=L_i+L_j$. The coisotropic subalgebra $\goh$ that we obtain, for a fixed i and j, in $\gog$ is generated by :
\[ \{ X_{ik} , Y_{kj}\}_{i<k \neq j},\{ X_{jk} , Y_{ki}\}_{j<k}, Y_{ij},\ H_i - H_j \]
without loosing any generality one can restrict the study to i=1. But we will distinct two cases. \newline

$\bullet$ The first one if $j=n$ for which we will obtain  the coisotropic subalgebra $\goh$ in $\gog$ generated by:
\begin{align*} (a) & \ h_1 + h_2 + \cdots + h_{n-1}, \ e_1, \ \crc{e_1,e_2}, \ \crc{\crc{e_1,e_2},e_3},\ \dots \ \crc{\crc{e_1,e_2},\dots,e_{n-2}},\\
(b) & \ e_n, \crc{\crc{\crc{e_1,e_2}\dots,e_{n-1}},e_n}, \\
(c) & \ y=\crc{e_n,\crc{e_n,e_{n-1}}}, \ \crc{y,e_{n-2}}, \dots , \crc{\crc{y,e_{n-2}},\dots,e_1} \end{align*}

\begin{equation*} \small \left( \begin{array}{ccccc|ccccc|c} 
                     a_0 & a_1 &\cdots & a_{n\text{-}2} & 0        & 0         &     0     & \cdots & 0         &\text{-}c_1 & b_2\\
                           & 0     &  \cdots & 0        &0        &  0         &          &            &     0       & \text{-}c_2 & 0  \\
                           &        & \ddots &  \vdots &\vdots & \vdots &           &           & \vdots       & \vdots & \vdots\\
                           &       &              & 0         &   0      & 0         & 0 &  \cdots  &  0        & \text{-}c_{n\text{-}1} & 0 \\  
                          &        &              &           & \text{-}a_0    & c_1  & c_2     & \cdots & c_{n\text{-}1} & 0 & b_1 \\
\hline 
                          &        &             &            &            & \text{-}a_0   &              &            &             &        &\\
                          &        &             &            &            & \text{-}a_1   &   0         &            &             &        &\\
                           &        &             &            &           & \vdots  &   \vdots &   \ddots &             &    &    \\
                            &        &             &            &          & \text{-}a_{n\text{-}2}&       0    &    \cdots&   0      &        &\\
                             &        &             &            &            & 0        & 0  & \cdots  &    0    & a_0          &\\
\hline                &        &             &            &            & \text{-}b_2       & 0  & \cdots  &    0    & \text{-}b_1          & 0\\
 \end{array} \right) \end{equation*}
 it's counterpart $B_h$ in $U_q(\goso(2n+1))$ is generated by:
\small{\begin{align*} (a) &\ \  K_1\cdots .K_n, \ E_1 ,\ \cro{E_1}{E_2}_{q^2}, \ \cro{\cro{E_1}{E_2}_{q^2}}{E_3}_{q^2} , \ \dots \ , \  \cro{\cro{E_1}{E_2}_{q^2}}{\dots,E_{n-2}}_{q^2} \\
 (b) &  \ \ E_n , \ \crc{E_n,\cro{\cro{E_1}{E_2}_{q^2}}{\dots,E_{n-1}}_{q^2}}_{q^2} \\
(c) & \ \ \crc{E_n,\crc{E_n,E_{n-1}}_{q^2}}, \cro{\crc{E_n,\crc{E_n,E_{n-1}}_{q^2}}}{E_{n-2}}_{q^2}, \ \dots \ , \  \cro{\crc{E_n,\crc{E_n,E_{n-1}}_{q^2}}}{\dots,E_1}_{q^2} \end{align*}}
\begin{Prop} 
the subalgebra $B_h$ is a left coideal in $U_q(\goso(2n+1))$
\end{Prop}
\begin{proof}
For the first set of generators, it is like we always do. The second set of generators is trivial by considering the fact that:
\[ \Delta( \cro{E_n}{E_{n-1}}_{q^2} ) = 1 \otimes \cro{E_n}{E_{n-1}}_{q^2} + E_n \otimes \cro{K_n}{E_{n-1}}_{q^2} + \cro{E_n}{E_{n-1}}_{q^2} \otimes K_nK_{n-1} \]
Now for the third set of generators, we will compute $\Delta(\crc{E_n,\crc{E_n,E_{n-1}}_{q^2}})$, only the term  $\cro{1 \otimes E_n}{\cro{E_n}{E_{n-1}}_{q^2} \otimes K_nK_{n-1}} $ will be an obstruction. But we can see that $\cro{E_n}{K_nK_{n-1}} = 0$ implying that $\cro{1 \otimes E_n}{\cro{E_n}{E_{n-1}}_{q^2} \otimes K_nK_{n-1}} =0$. Meaning that:
\begin{align*}   \Delta(\crc{E_n,\crqq{E_n,E_{n-1}}})  = & 1 \otimes  \crc{E_n,\crqq{E_n,E_{n-1}}} \\
                                                              & + \ \ E_n \otimes \crc{E_n,\crqq{K_n,E_{n-1}}} + \crc{K_n,\crqq{E_n,E_{n-1}}} \\
                                                              & + \ \ E_n^2 \otimes \crc{K_n,\crqq{K_n,E_{n-1}}} \\
                                                              & + \ \ \crqq{E_n,\crqq{E_n,E_{n-1}}} \otimes K_n^2K_{n-1} \end{align*}  
The last set of generators is done by computing the generators one by one. \newline
one can check by computation that for ${}_nY_j = \crqq{\crqq{\crc{E_n,\crqq{E_n,E_{n-1}}},E_{n-2}}\dots,E_j}$
\begin{align*}  \Delta({}_nY_j)  = & 1 \otimes {}_nY_j  +  E_n \otimes \Big( \crqq{\crqq{\crc{K_n , \crqq{E_n,E_{n-1}}},E_{n-2}}\dots,E_j}  \\
& \ \ \ \ \ \ \ \ \ \ \ \ \ \ \ \ \ \ \ \ \   +  \crqq{\crqq{\crc{E_n, \crqq{K_n,E_{n-1}}},E_{n-2}}\dots,E_j} \Big) \\  
& + \ \ E_n^2 \otimes \crqq{\crqq{\crc{K_n, \crqq{K_n,E_{n-1}}},E_{n-2}}\dots,E_j} +  {}_nY_{n-1} \otimes \crqq{\crqq{K_n^2K_{n-1},E_{n-2}}\dots,E_j} \\
& + \ \ {}_nY_{n-2} \otimes \crqq{\crqq{K_n^2K_{n-1}K_{n-2},E_{n-3}}\dots,E_j}  + \  \cdots \ +  {}_nY_j \otimes   K_n^2K_{n-1}K_{n-2} \cdots K_j \end{align*}      \end{proof}
\begin{Theo} 
$B_h$ is a quantization of $\goh$.
\end{Theo}
\begin{proof} we will prove that $B_h$ is a flat deformation, by computation.\newline

$\bullet$ $A_1,A_2 \in ((a),(a))$, the demonstration is the same as in $\gosl(n+1)$ (with all the bracket becoming $q^2$).  \newline

$\bullet$ $A_1,A_2 \in ((a),(b))$, it is trivial, as we have that $X_k$ commutes with $E_n$ and that it $\crc{X_k,X_{n-1} }_{q^{-2}} = 0.$. \newline

$\bullet$ $A_1,A_2 \in ((a),(c))$, it is done exactly the same as in $\goso{2n}$ except that all the brackets are $q^2$. \newline

$\bullet$ $A_1,A_2 \in ((b),(b))$, we only need to consider $\crc{E_n,\crc{E_n,\crc{E_n,E_{n-1}}_{q^2}} }_{q^{-2}}$ which is zero by using the Serre relations. \newline
 
$\bullet$ $A_1,A_2 \in ((b),(c))$, it is trivial to see that $\crc{E_n,_nY_j }_{q^{-2}} = 0$. Therefore, we only need to verify that \[\crc{\crc{E_n,\crc{E_n,E_{n-1}}_{q^2}},\crc{\crc{E_n,\crc{E_n,E_{n-1}}_{q^2}},E_{n-2}}_{q^2} }_{q^{-2}} = 0 \]
To prove this we will use the following method. First, let's set $A=E_n$, $B=E_{n-1}$ and $C=\crc{E_{n-1},E_{n-2}}$.
We have
\[ \crc{\crc{E_n,\crc{E_n,E_{n-1}}_{q^2}},\crc{\crc{E_n,\crc{E_n,E_{n-1}}_{q^2}},E_{n-2}}_{q^2} }_{q^{-2}} = \crc{\crc{A,\crc{A,B}_{q^2}},\crc{A,\crc{A,C}_{q^2}} }_{q^{-2}}.\]
Furthermore, we have the following relations
\small{\begin{align*} R_B :=& \crc{A,\crc{A,\crc{A,B}_{q^2}} }_{q^{-2}} = 0 \\
                      R_C :=& \crc{A,\crc{A,\crc{A,C}_{q^2}} }_{q^{-2}} = 0 \\
                   R_{BAC}:=& \crc{B,\crc{A,C}_{q^2}} = 0 \end{align*} }
We will then prove that
\small{\begin{align*}
\crc{\crc{A,\crc{A,B}_{q^2}},\crc{A,\crc{A,C}_{q^2}}}_{q^{-2}} =&  aR_BAC + bR_BCA + cAR_BC + dBR_CA + eBAR_C+fABR_C \\
                                                       &+ \ a'R_CAB + b' R_CBA + c'AR_CB + d'CR_BA + e'CAR_B+f'ACR_B \\
                                                       &+ \ gR_{BAC}AAA + hAR_{BAC}AA+ iAAR_{BAC}A + jAAAR_{BAC} \end{align*} }
We obtain a linear system. We solve it and find one solution:
\[ \begin{array}{cccccc} a=0 & b=-\frac{1}{q^2+q^4+1} & c=\frac{q^2}{q^2+q^4+1}& d = \frac{q^4+q^2}{q^2+q^4+1}& e = q^2 & f = -\frac{q^6+2q^4+q^2+1}{q^2+q^4+1} \\
                         a'=1 & b' = -\frac{q^6+q^4+2q^2+1}{q^2+q^4+1} & c'= \frac{q^4+q^2}{q^2+q^4+1}& d'= \frac{q^4}{q^2+q^4+1} & e'=0 & f'= -\frac{q^6}{q^2+q^4+1} \\
                         & g=-1 & h = \frac{1+q^2+q^4}{q^2} & i = - \frac{1+q^2+q^4}{q^2} & j := 1 & \end{array} \]

$\bullet$ $A_1,A_2 \in ((c),(c))$, we need here to compute $\crc{ _nY_k, _nY_l }_{q^{-2}}$ with k < l. But by using the proof in $\gosl(n+1)$, we can see that for $n-2 \ge i \ge k$ , we have that $E_i$ commutes with $_nY_l$ and therefore we have:
\[ \crc{ _nY_k, _nY_l}_{q^{-2}} = \crc{\crc{\crc{ _nY_{n-1},_nY_l}_{q^{-2}},E_{n-2}}_{q^2},\dots,E_k}_{q^2} \]
which is zero considering the last proof. \newline

$\bullet$ Of course like the preceding proof, we have that for all E generators in $B_h$, there exist $l \in \mathbb{N}$ such that :  \[ \cro{K_1\cdots K_{j-1}}{E} = (1 - q^l) K_1\cdots K_{j-1} E.\]
By using the proposition \ref{LemQ}, we finish our proof.
\end{proof}

$\bullet$ The second one if $j \neq n$, will be more complicated.
First of all, the candidate $B_h$ will be generated by :
\small{\begin{align*} (a) & \ \  K_1\cdots K_{j-1},\ E_1 ,\ \crqq{E_1,E_2}, \ \crqq{\crqq{E_1,E_2},E_3} , \ \dots \ , \  \crqq{\crqq{E_1,E_2}\dots,E_{j-2}} \\
(b) & \ \  E_j ,\ \crqq{E_j,E_{j+1}}, \ \dots \ , \ {}_jX_{n-1} = \crqq{ {}_jX_{n-2},E_{n-1}} \\
(c) & \ \ \crqq{E_j,T} , \ \crqq{\crqq{E_j,E_{j+1}},T}, \ \dots \ , \  \crqq{\crqq{\crqq{E_j,E_{j+1}}\dots,E_{n-1}},T} \\
 (d) & \ \ {}_jX_n  = \crqq{{}_jX_{n-1},E_n}, \ {}_jY_n = \crc{ {}_jX_n,E_{n}} , \ \crqq{ {}_jY_n,E_{n-2}} , \ \dots \ , \  {}_jY_{j+1} = \crqq{ {}_jY_{j+2},E_{j+1}} \\
 (e) & \ \ \crqq{ {}_jX_n,T} , \ \crqq{ {}_jY_n,T}, \ \crqq{\crqq{ {}_jY_n,E_{n-2}},T}, \ \dots \ , \
\crqq{ {}_jY_{j+1},T} \\
(f) & \ \crqq{ {}_jY_{j+1},\crqq{E_j,E_{j-1}}}, \dots \ , \ \crqq{\crqq{\crqq{ {}_jY_{j+1},\crqq{E_j,E_{j-1}}},E_{j-2}}\dots,E_1} \end{align*}}

\begin{equation*} \scriptsize  \left(\begin{array}{cccccccc|cccccccc|c}
a_0& a_1 &\cdots& a_{j\text{-}2}& 0      &  c_j    &\cdots & c_{n\text{-}1}    & 0    &0  &\cdots&  0   & \text{-}f_1 & e_j &\cdots& e_{n\text{-}1}  & e_n \\
     & 0     &\cdots& 0        &0      &  0     &\cdots &   0      &0  &     &       &     0      & \text{-}f_2 & 0 & \cdots  & 0   & 0  \\
    &        &\ddots&\vdots &\vdots&\vdots&   & \vdots& \vdots&    &     &  \vdots&\vdots&\vdots& & \vdots & \vdots\\
  &        &          & 0         &   0      & 0       & \cdots &    0  &      0   & 0  &\cdots& 0   &\text{-}f_{j\text{-}1}&  0 &\cdots& 0 & 0\\ 
   &        &          &           & \text{-}a_0  & b_j     &  \cdots & b_{n\text{-}1} & f_1 &f_2& \cdots& f_{j\text{-}1} & 0 & d_j &\cdots& d_{n\text{-}1} & d_n \\
   &        &          &           &          &  0       & \cdots & 0       & \text{-}e_j &    0  &\cdots   & 0          & \text{-}d_j& 0 & \cdots & 0 & 0\\
   &        &          &           &          &          &  \ddots&\vdots&\vdots& \vdots &            &\vdots& \vdots&\vdots&  & \vdots & \vdots \\
 &        &          &           &          &          &                & 0      & \text{-}e_{n\text{-1}}  & 0    &\cdots  & 0          & \text{-}d_{n\text{-1}}   & 0 &\cdots& 0  & 0\\
\hline 
 &        &          &           &          &          &     &     &   \text{-}a_0 & & & & & & &  &\\
  &        &          &           &          &          &    &     &  \text{-}a_1 & 0 & & & & & & &\\
 &        &          &           &          &          &   &     & \vdots&\vdots&\ddots&        &   &  &  & &\\
 &        &          &           &          &          &    &     &\text{-}a_{j\text{-}2}& 0       & \cdots& 0    &     &  & & &\\
  &        &          &           &          &          &   &     &  0       & 0       & \cdots & 0    & a_0& & &  &\\
  &        &          &           &          &          &   &     &\text{-}c_{j}  & 0        & \cdots& 0      &\text{-}b_j& 0 & & &\\
  &        &          &           &          &          &   &     & \vdots& \vdots&          &\vdots&\vdots&\vdots &\ddots & &\\
  &        &          &           &          &          &   &     & \text{-}c_{n\text{-}1}& 0& \cdots& 0 & \text{-}b_{n\text{-}1}& 0& \cdots& 0 & \\
\hline   &        &          &           &          &          &   &     & \text{-}e_n& 0& \cdots& 0 & \text{-}d_n& 0& \cdots& 0 & \\
\end{array} \right) \end{equation*}

\begin{Prop} 
The subalgebra $B_h$ is a left coideal in $U_q(\goso(2n+1))$
\end{Prop}

\begin{proof}
The proof for the first three set of generators is exactly the same as in $\goso(2n)$. For the fourth set of generators, it is exactly like the previous example in $\goso(2n+1)$. Let's set ${}_jY_k = \crc{\crc{\crc{ {}_jX_n,E_n},E_{n-1}}\dots,E_k}$ for $k \ge j+1$
\small{\begin{align*} \Delta({}_jX_n)  = & 1 \otimes {}_jX_n   +  E_j \otimes \cro{\cro{\cro{K_j}{E_{j+1}}_{q^2}}{\dots,E_{n-1}}_{q^2}}{E_n}_{q^2}  \\
                                     & +  {}_jX_{j+1} \otimes \cro{\cro{\cro{K_jK_{j+1}}{E_{j+2}}_{q^2}}{\dots,E_{n-1}}_{q^2}}{E_n}_{q^2} \\
                                    & +\  \cdots \ +  {}_jX_{n-1} \otimes \cro{{}_jK_{n-1}}{E_n}_{q^2}   +  {}_jX_n \otimes {}_jK_n \end{align*}}
for ${}_jY_n$
\[ \Delta({}_jY_n)= \cro{ \Delta({}_jX_n)}{1 \otimes E_n + E_n \otimes K_n}  \]
We only need to look at $\cro{ \Delta({}_jX_n)}{E_n \otimes K_n}$.  It is easy to see that for $j \le k \le n-2$, we have that $E_n$ commutes with ${}_jX_k$ and that $\crc{\crc{\cro{K_jK_{j+1} \cdots K_{k}}{E_{k+1}}_{q^2},E_n}_{q^2},K_n}_{q^2} =0$. Also the last term $\crc{ {}_jX_n \otimes K_jK_{j+1} \cdots K_{n-1}K_n,E_n \otimes K_n}$ is not an obstruction. We need to consider the term :
\small{\begin{align*}
\crc{ {}_jX_{n-1} \otimes \cro{{}_jK_{n-1}}{E_n}_{q^2},E_n \otimes K_n }  = & {}_jX_{n-1} E_n \otimes \cro{{}_jK_{n-1}}{E_n}_{q^2} K_n  \\  & - E_n \ {}_jX_{n-1} \otimes K_n \cro{{}_jK_{n-1}}{E_n}_{q^2} \\
                                                                            = & \ {}_jX_{n-1} E_n - q^2 E_n \ {}_jX_{n-1} \otimes \cro{{}_jK_{n-1}}{E_n}_{q^2} K_n \\
                                                                            = & \ {}_jX_n \otimes \cro{{}_jK_{n-1}}{E_n}_{q^2} K_n \end{align*}}
In the end, we find for ${}_jY_n$,
\begin{align*} \Delta({}_jY_n)  = & 1 \otimes {}_jY_n   + E_j \otimes  \cro{\cro{\cro{\cro{K_j}{E_{j+1}}_{q^2}}{\dots,E_{n-1}}_{q^2}}{E_n}_{q^2}}{E_n} \\
                            & + {}_jX_{j+1} \otimes \cro{\cro{\cro{\cro{K_jK_{j+1}}{E_{j+2}}_{q^2}}{\dots,E_{n-1}}_{q^2}}{E_n}_{q^2}}{E_n} \\
                            & + \  \cdots  \ + {}_jX_{n-1} \otimes  \cro{\cro{{}_jK_{n-1}}{E_n}_{q^2}}{E_n}\\
                            & + {}_jX_n \otimes \left( \cro{{}_jK_{n}}{E_n} + \cro{{}_jK_{n-1}}{E_n}K_n \right)\\
                            & + {}_jY_n \otimes {}_jK_{n} K_n   \end{align*}
The rest of the proof consists of the same demonstration as in $\goso(2n)$. \end{proof}
\begin{Theo} 
$B_h$ is a quantization of $\goh$.
\end{Theo}
\begin{proof} \newline
The proof here is done like the previous one (a mix between the last one and the one of $\goso(2n)$).
\end{proof}

\section{Exceptional Lie bialgebras}

We will here construct the example on the Lie bialgebras of type $G_2$. The case of $F_4$ is trivial in this case because we have that none of the positive roots verifies the property. Therefore, we cannot construct an example.\newline
Now let's focus on the case of $G_2$. The roots are given by $R = \{\pm L_1, \pm \sqrt{3} L_2, \pm \frac12 L_1 \pm \frac{\sqrt3}2 L_2, \pm \frac32 \pm \frac{\sqrt3}2 L_2 \}$, the simple roots are $\alpha_1 = L_1$  and $\alpha_2 = - \frac32 + \frac{\sqrt3}2 L_2$.The roots that satisfy the assumption are $\pm \sqrt{3} L_2$ and $\pm \frac32 \pm \frac{\sqrt3}2 L_2$.
The root space of $L_1$ is given by $x_1 = e_1$ and $y_1 = f_1$, for  $\frac32 + \frac{\sqrt3}2 L_2$ it is given by $x_2 = e_2$ and $y_2 = f_2$, for $-\frac12 L_1 + \frac{\sqrt3}2 L_2 = \alpha_1 + \alpha_2$ it is given by $x_3 = \cro{e_1}{e_2}$ and $y_3 = \cro{f_1}{f_2}$, for $\frac12 L_1 + \frac{\sqrt3}2 L_2 = \alpha_1 + \alpha_1 + \alpha_2$ it is given by $x_4 = \cro{e_1}{x_3}$ and $y_4 = \cro{f_1}{y_3}$, for $\frac32 L_1 + \frac{\sqrt3}2 L_2 = \alpha_1 + \alpha_1 + \alpha_1 + \alpha_2$ it is given by  $x_5 = \cro{e_1}{x_4}$ and $y_5 = \cro{f_1}{y_4}$, and finally for $\sqrt3 L_2 = \alpha_2 +\alpha_1 + \alpha_1 + \alpha_1 + \alpha_2$ it is given by $x_6 = \cro{e_2}{x_5}$ and $y_6 = \cro{f_2}{y_5}$. But for the computation to be easier, we will apply the changes that were done by Fulton and Harris.
We need to compute the r-matrix:
\[ \pi = \frac1{24} \left( x_1 \wedge y_1 + x_3 \wedge y_3 + x_4 \wedge y_4 \right) \frac18 \left( x_2 \wedge y_2 + x_5 \wedge y_5 + x_6 \wedge y_6 \right) \]
we fix $\beta = \alpha_2$ therefore we compute the bracket:
\[ \cro{e_2}{\pi} = \lambda (e_2 \wedge h_2) \]
The coisotropic subalgebra is spanned by : $e_2$ and $h_1 + h_2$. This example is trivial.
We fix $\beta = \frac32 + \frac{\sqrt3}2 L_2$ therefore the bracket gives:
\[ \cro{x_5}{\pi} = 2 x_1 \wedge x_4 + x_5 \wedge h_1 + h_2 \]
Therefore, the coisotropic subalgebra $\goh$ is spanned by
\[ h_1 + h_2,\ x_1 ,\ x_4 ,\ x_5 \]
and its quantum counterpart $B_h$
\[ K_1K_2,\ E_1,\ X=\cro{\cro{E_1}{E_2}_{q^3}}{E_1}_{q^{-1}},\ Y=\cro{\cro{\cro{E_1}{E_2}_{q^3}}{E_1}_{q^{-1}}}{E_1}_q \]
\begin{Prop}
$B_h$ is a left coideal of $U_q(\gog)$
\end{Prop}
\begin{proof}
We have to check that $\Delta(B_h) \subset B_h \otimes U_q(\gog)$. It is direct for $K_1K_2$ and  $E_1$.
We have to check it for $\cro{\cro{E_1}{E_2}}{E_1}$
\[ \Delta(\cro{E_1}{E_2}) = 1 \otimes \cro{E_1}{E_2}_{q^3} + E_1 \otimes \cro{K_1}{E_2}_{q^3} + \cro{E_1}{E_2}_{q^3} \otimes K_1K_2 \]
and therefore
\begin{align*}\Delta(X) =& 1   \otimes \cro{\cro{E_1}{E_2}_{q^3}}{E_1}_{q^{-1}}  + \cro{E_1}{E_2}_{q^3} \otimes \cro{K_1K_2}{E_1}_{q^{-1}} \\ 
                                                        & +\    E_1 \left( \otimes \cro{\cro{K_1}{E_2}_{q^3}}{E_1}_{q^{-1}} + \cro{\cro{E_1}{E_2}_{q^3}}{K_1}_{q^{-1}}  \right)  \\
                                                       & +\  E_1^2 \otimes \cro{\cro{K_1}{E_2}_{q^3}}{K_1}_{q^{-1}}  + X \otimes K_1^2K_2 \end{align*}
The only term that need to disapear is $\cro{E_1}{E_2}_{q^3}$ , but we have that $\cro{K_1K_2}{E_1}_{q^{-1}}=0$. Wich justifies the use of the $q^{-1}$ bracket. The last one is given directly by the fact that both $\Delta(X)$ and $\Delta(E_1)$ are in $B_h \otimes U_q(\gog)$.
Thus proving our proposition.
\end{proof}

\begin{Theo} 
$B_h$ is a quantization of $\goh$.
\end{Theo}

\begin{proof}
Using the lemma \ref{LemQ}, we need to prove that for all generators $A_1$, $A_2$ we have that $\cro{A_1}{A_2}$ is composed of elements either well ordered, of degree 1 (the same as well ordered here) or of valuation on h greater than $A_1A_2$. \newline For $A_1 = E_1$, we have that $\cro{E_1}{\cro{\cro{E_1}{E_2}}{E_1}}$ is a generator and that $\cro{E_1}{\cro{\cro{\cro{E_1}{E_2}_{q^3}}{E_1}_{q^{-1}}}{E_1}_q }_{q^{-3}}$ is zero by using the Serre relation which is:
\[ \crc{E_1,\crc{E_1,\crc{E_1,\crc{E_1,E_2}_{q^3}}_q}_{q^{-1}} }_{q^{-3}} = 0 .\]
Therefore only one bracket remains, that is $\cro{\cro{\cro{E_1}{E_2}_{q^3}}{E_1}_{q^{-1}}}{\cro{\cro{\cro{E_1}{E_2}_{q^3}}{E_1}_{q^{-1}}}{E_1}_q}$ which is also zero by using the two Serre relations and solving a linear system using those equations.
Of course like the preceding proof, we have  that for all A generators in $B_h$, there exist $l \in \mathbb{N}$ such that $\cro{K_1K_2}{A} = (1 - q^l) K_1K_2 A$. \newline
Therefore, by using the proposition \ref{PropQ}, we finish the demonstration.
\end{proof}

Finally, for $\beta =  \sqrt{3} L_2$, we have:
\[ \cro{x_6}{\pi} = 2 x_2 \wedge x_5 + 2 x_3 \wedge x_4 + x_6 \wedge h_1 + 2 h_2 \]
Therefore, the coisotropic subalgebra $\goh$ is spanned by
\[ h_1 + 2h_2,\ x_2 ,\ x_3 ,\ x_4, \ x_5,\ x_6 \]
and its quantum counterpart

\[ K_1K_2^2,\ E_2,\ X = \cro{E_2}{E_1}_{q^3} ,\ Y=\cro{X}{E_1}_q,\ Z=\crc{Y,E_1}_{q^{-1}}, T=\crc{Z,E_2} \]

\begin{Prop}
$B_h$ is a left coideal of $U_q(\gog)$
\end{Prop}

\begin{proof}
We have to check that $\Delta(B_h) \subset B_h \otimes U_q(\gog)$. It is direct for $K_1K_2^2$ and  $E_2$. After we chose the generator so that $E_1$ vanishes on the left side of the tensor.
\[ \Delta(\cro{E_2}{E_1}_{q^3}) = 1 \otimes \cro{E_2}{E_1}_{q^3} + E_1 \otimes \cro{E_2}{K_1}_{q^3} + E_2 \otimes \cro{K_2}{E_1}_{q^3} + \cro{E_1}{E_2}_{q^3} \otimes K_1K_2 \]
 we have that $\cro{E_2}{K_1}_{q^3} = 0$. Therefore for $X = \crc{E_2,E_1}_{q^3}$
\[ \Delta(X)= 1 \otimes X + E_2 \otimes  \cro{K_2}{E_1} + X \otimes K_1K_2 \]
for the next generator a simple computation can show that we need to use q bracket to get rid of the term  $E_1 \otimes \cro{X}{K_1}_q$ as $\cro{X}{K_1}_q = 0$.
\begin{align*}\Delta(Y) =&  1 \otimes Y  +  E_2 \otimes  \cro{\cro{K_2}{E_1}_{q^3}}{E_1}_q  +   E_1 \otimes \cro{X}{K_1}_q  \\
                              &+\ X \otimes \left( \cro{K_1K_2}{E_1}_q + \cro{K_2}{E_1}_{q^3}K_1\right)  + Y \otimes K_1^2K_2 
                       \end{align*}
                         For Z as for Y, a simple computation and reordering of terms, show that we need to consider the $q^{-1}$ bracket.
\begin{align*} \Delta(Z) =& 1 \otimes Z + E_2 \otimes \crc{\cro{\cro{K_2}{E_1}_{q^3}}{E_1}_q,E_1}_{q^{-1}} \\
                                & +\ X  \otimes  \left(\crc{\cro{K_2}{E_1}_{q^3}K_1,E_1}_{q^{-1}} + \crc{\cro{K_1K_2}{E_1}_q,E_1}_{q^{-1}} + \cro{\cro{K_2}{E_1}_{q^3}}{E_1}_qK_1 \right) \\
                                & +\ Y  \otimes\left(  \crc{K_1^2K_2,E_1}_{q^{-1}}  -\cro{K_2}{E_1}_{q^3}K_1K_1\right) \\
                                & +\ Z  \otimes  K_1^3K_2
                                 \end{align*}
therefore $\Delta(Z)$ is in $B_h \otimes U_q(\gog)$ and at the same time this proves it for $T$ as $\Delta(E_2)$ and $\Delta(Z)$ are in $B_h \otimes U_q(\gog)$.
\end{proof}

\begin{Theo} 
$B_h$ is a quantization of $\goh$.
\end{Theo}

\begin{proof}
Using the lemma \ref{LemQ}, we need to prove that for all generators $a_1$, $a_2$ we have that $\cro{a_1}{a_2}$ is composed of elements either well ordered, of degree 1 (the same as well ordered here) or of valuation on h greater than $a_1a_2$. \newline
- For $A_1 = E_2$, we have to compute $\crc{E_2,\crc{E_2,E_1}_{q^3} }_{q^{-3}}$ which is zero because it is the Serre relation between $E_2$ and $E_1$.
\[ \crc{E_2,\crc{\crc{E_2,E_1}_{q^3},E_1}_q}_{0} = q^{-3} \crc{\crc{E_2,E_1}_{q^3},\crc{E_2,E_1}_{q^3}}_4 = q^{-3}(1-q^4) \crc{E_2,E_1}_{q^3}^2.   \]
then we have to compute:
\[ \crc{E_2, \crc{\crc{\crc{E_2,E_1}_{q^3},E_1}_q,E_1}_{q^{-1}}}_{0} = -T \]
and finally :
\[ \crc{E_2, T} = 0 + h*C \]
We prove that by using the Serre relations $R_2 = \crc{E_2,\crc{E_2,E_1}_{q^3} }_{q^{-3}}$ and elements in $B_h$ obtained by combining the elements $E_2, T$ or the elements $E_2,E_2,\crc{\crc{\crc{E_2,E_1}_{q^3},E_1}_q,E_1}_{q^{-1}}$ or the elements $E_2, \crc{E_2,E_1}_{q^3}, \crc{\crc{E_2,E_1}_{q^3},E_1}_q$ or $\crc{E_2,E_1}_{q^3}^3$. This allows us to have a linear system of 20 equations with 24 undetermined with some constraints on some undetermined (we want that h divides some of them). \newline
- For $A_1 = X$, we have to compute $\crc{X,Y}$ by using the same demonstration as $\crc{E_2,Z} = - T$ and $\crc{X,Z}$ and $\crc{X,T}$ by using the same demonstration as $\crc{E_2, T}$. \newline
- For $A_1 = Y$, we have to compute $ \crc{Y,Z}$ and $\crc{Y,T}$ which are still the same as $\crc{E_2,T}$. \newline
- For $A_1 = Z$, we finally have to compute $\crc{Z,T}$.
Of course like the preceding proof, we have  that for all E generators in $B_h$, there exist $l \in \mathbb{N}$ such that $\cro{K_1K_2^2}{E} = (1 - q^l) K_1K_2^2 E$. \newline
Therefore, by using the proposition \ref{PropQ}, we finish the demonstration.
\end{proof}

We will give some example in the case of $E_6$. We proceed in the exact same way as before. It will be really long to explicit every step for $E_6$ because of the number of generators and the fact that all the roots verify the assumption. Therefore, we will directly give the generators of the quantum coisotropic subalgebras. But first we need to find the r-matrix. The r-matrix that we need requires to calculate the Killing form. By using the fact that in $E_6$, all the roots are of equal length and that we can set for every root $\alpha$ that $\cro{e_{\alpha}}{f_{\alpha}}=-h_{\alpha}$ , $\cro{h_{\alpha}}{e_{\alpha}}=e_{\alpha}$, $\cro{h_{\alpha}}{f_{\alpha}}=-f_{\alpha}$. Then, we have that the Killing form $K(e_{\alpha},f_{\alpha})= \frac12K(h_{\alpha},h_{\alpha})$. And we have that if the root system is irreducible and that all the roots are of equal length then $K(h_{\alpha},h_{\alpha})=4k$ where k is the coxeter number. Therefore we have:
\[ \pi = \frac1{2k} \sum_{\alpha \in R^+} e_{\alpha} \wedge f_{\alpha} \]
We just need to take for $\alpha = \alpha_{i_1} + \cdots + \alpha_{i_r}$:
\[ e_{\alpha} = \cro{\cro{e_{\alpha_{i_1}}}{e_{\alpha_{i_{2}}}}}{\dots,e_{\alpha_{i_r}}} \in \gog^{\alpha} \]
and
\[ f_{\alpha} = (-1)^{r} \cro{\cro{f_{\alpha_{i_1}}}{f_{\alpha_{i_{2}}}}}{\dots,f_{\alpha_{i_r}}} \in \gog^{-\alpha} \]
Now we compute for every root the bracket of $e_{\alpha}$ and $\pi$ to find the coisotropic subalgebras.

\begin{Rem}
The same method can be use for $E_7$ and $E_8$, because all the roots are of the same length. For more information and a demonstration of this method we refer to \cite{[B481]} and \cite{[B781]}.
\end{Rem}
With those two tables, we have 36 examples of coisotropic subalgebras (by using the fact that for each * we can construct a symmetric coisotropic subalgebra by replacing $E_1$ by $E_6$ and $E_3$ by $E_5$). The proofs are similar to the one done in the case of $\goso(2n)$.

\scriptsize{
\begin{tabular}{|l|l|}
\hline  Roots & Candidate $B_h$ in $U_q(E_6)$ \\

\hline  $\alpha_i$ & $E_i$, $K_i$ \\

\hline  $\alpha_1+\alpha_3$ *& $E_1$, $E_3$, $\cro{E_1}{E_3}$, $K_1K_3$ \\

\hline  $\alpha_3+\alpha_4$ *& $E_3$, $E_4$, $\cro{E_3}{E_4}$, $K_3K_4$ \\

\hline  $\alpha_2+\alpha_4$ *& $E_2$, $E_4$, $\cro{E_2}{E_4}$, $K_2K_4$ \\

\hline  $\alpha_1 + \alpha_3 + \alpha_4$ *& $E_1$, $E_4$, $\cro{E_1}{E_3}$, $\cro{E_4}{E_3}$, $\cro{\cro{E_1}{E_3}}{E_4}$,    $K_1K_3K_4$ \\

\hline  $\alpha_3 + \alpha_4 + \alpha_5$ & $E_3$, $E_5$, $\cro{E_3}{E_4}$, $\cro{E_5}{E_4}$, $\cro{\cro{E_3}{E_4}}{E_5}$,    $K_3K_4K_5$ \\

\hline  $\alpha_3 + \alpha_4 + \alpha_2$ *& $E_3$, $E_2$, $\cro{E_3}{E_4}$, $\cro{E_2}{E_4}$, $\cro{\cro{E_3}{E_4}}{E_2}$,    $K_3K_4K_2$ \\

\hline  $\alpha_1 + \alpha_3 + \alpha_4 $& $E_1$, $E_5$, $\cro{E_1}{E_3}$, $\cro{E_5}{E_4}$,  $\cro{\cro{E_1}{E_3}}{E_4}$, $\cro{\cro{E_5}{E_4}}{E_3}$, \\ $ \ \ \ \ + \alpha_5$ *
 &  $\cro{\cro{\cro{E_1}{E_3}}{E_4}}{E_5}$,     $K_1K_3K_4K_5$ \\

\hline  $\alpha_1 + \alpha_3 + \alpha_4 $ & $E_1$, $E_2$, $\cro{E_1}{E_3}$, $\cro{E_2}{E_4}$,  $\cro{\cro{E_1}{E_3}}{E_4}$, $\cro{\cro{E_2}{E_4}}{E_3}$,\\ $ \ \ \ \ + \alpha_2$ *
& $\cro{\cro{\cro{E_1}{E_3}}{E_4}}{E_2}$,     $K_1K_3K_4K_2$ \\

\hline  $ \alpha_3 + \alpha_4 + \alpha_5 $ &  $E_3$, $E_2$, $E_5$, $\cro{E_5}{\cro{E_2}{E_4}}$, $\cro{E_5}{\cro{E_3}{E_4}}$, \\
$ \ \ \ \ + \alpha_2$ & $\cro{E_2}{\cro{E_3}{E_4}}$, $\cro{E_5}{\cro{E_2}{\cro{E_3}{E_4}}}$, $K_3K_4K_5K_2$ \\

\hline  $\alpha_1 + \alpha_3 + \alpha_4 $ & $E_1$, $E_6$, $\cro{E_1}{E_3}$, $\cro{E_6}{E_5}$,  $\cro{\cro{E_1}{E_3}}{E_4}$,\\ 
$ \ \ \ \ + \alpha_5 + \alpha_6$ & $\cro{\cro{E_6}{E_5}}{E_4}$, $\cro{\cro{\cro{E_1}{E_3}}{E_4}}{E_5}$, $\cro{\cro{\cro{E_6}{E_5}}{E_4}}{E_3}$, \\ &$\cro{\cro{\cro{\cro{E_1}{E_3}}{E_4}}{E_5}}{E_6}$,     $K_1K_3K_4K_5K_6$ \\

\hline  $\alpha_1 + \alpha_3 + \alpha_4 $& $E_1$, $E_2$, $E_5$, $\cro{E_1}{E_3}$, $\cro{E_5}{\cro{E_2}{E_4}}$, $\cro{E_5}{\cro{\cro{E_1}{E_3}}{E_4}}$,\\ $ \ \ \ \ + \alpha_5 + \alpha_2$ *
& $\cro{E_2}{\cro{\cro{E_1}{E_3}}{E_4}}$, $\cro{E_5}{\cro{E_2}{\cro{\cro{E_1}{E_3}}{E_4}}}$, $K_1K_3K_4K_5K_2$ \\

\hline  $\alpha_1 + \alpha_3 + \alpha_4 $ & $E_1$, $E_2$, $E_6$, $\cro{E_1}{E_3}$, $\cro{E_6}{E_5}$,  $\cro{E_2}{\cro{\cro{E_1}{E_3}}{E_4}}$,\\ $ \ \ \ \ + \alpha_5 + \alpha_6 + \alpha_2$
& $\cro{E_2}{\cro{\cro{E_6}{E_5}}{E_4}}$, $\cro{E_2}{\cro{\cro{\cro{E_1}{E_3}}{E_4}}{E_5}}$,\\ & $\cro{E_6}{\cro{\cro{\cro{E_1}{E_3}}{E_4}}{E_5}}$,  $\cro{E_2}{\cro{\cro{\cro{E_6}{E_5}}{E_4}}{E_3}}$,\\& $\cro{E_2}{\cro{E_6}{\cro{\cro{\cro{E_1}{E_3}}{E_4}}{E_5}}}$, $K_1K_3K_4K_5K_6K_2$ \\

\hline $\alpha_3 + 2 \alpha_4 + \alpha_5$ & $E_4$, $\cro{E_4}{E_2}$, $\cro{E_4}{E_5}$, $\cro{E_4}{E_3}$, $\cro{\cro{E_4}{E_3}}{E_5}$,\\ $ \ \ \ \  + \alpha_2$
& $\cro{\cro{E_4}{E_5}}{E_2}$, $\cro{\cro{E_4}{E_3}}{E_2}$, $\cro{\cro{\cro{E_4}{E_3}}{E_5}}{E_2}$, \\
& $\cro{E_4}{\cro{\cro{\cro{E_4}{E_3}}{E_5}}{E_2}}$, $K_3K_4^2K_5K_2$ \\

\hline $\alpha_1 + \alpha_3 + 2 \alpha_4 $& $E_1$ , $E_4$, $\cro{E_4}{E_2}$, $\cro{E_4}{E_5}$, $\cro{\cro{E_4}{E_5}}{E_2}$, $\cro{E_4}{\cro{E_1}{E_3}}$, \\ $ \ \ \ \ + \alpha_5 + \alpha_2$ *&
$\cro{\cro{E_4}{E_2}}{\cro{E_1}{E_3}}$, $\cro{\cro{E_4}{E_5}}{\cro{E_1}{E_3}}$, $\cro{\cro{\cro{E_4}{E_5}}{E_2}}{\cro{E_1}{E_3}}$, \\&
$\cro{\cro{\cro{E_4}{E_5}}{E_2}}{\cro{E_4}{E_3}}$, $\cro{E_1}{\cro{\cro{\cro{E_4}{E_5}}{E_2}}{\cro{E_4}{E_3}}}$, $K_1K_3K_4^2K_5K_2$ \\

\hline $ \alpha_1 + 2 \alpha_3 + 2 \alpha_4 $& $E_3$, $\cro{E_3}{E_4}$, $\cro{E_3}{E_1}$,  $\cro{\cro{E_3}{E_4}}{E_5}$, $\cro{\cro{E_3}{E_4}}{E_2}$, \\$ \ \ \ \ + \alpha_5 + \alpha_2$ *& $\cro{\cro{E_3}{E_1}}{E_4}$,  $\cro{\cro{\cro{E_3}{E_1}}{E_4}}{E_5}$, $\cro{\cro{\cro{E_3}{E_1}}{E_4}}{E_2}$, \\& 
$\cro{\cro{\cro{E_3}{E_4}}{E_5}}{E_2}$, $\cro{\cro{\cro{\cro{E_3}{E_1}}{E_4}}{E_5}}{E_2}$, $\cro{\cro{\cro{\cro{E_3}{E_4}}{E_5}}{E_2}}{E_4}$, \\& $\cro{\cro{\cro{\cro{\cro{E_3}{E_1}}{E_4}}{E_5}}{E_2}}{E_4}$, $\cro{\cro{\cro{\cro{\cro{\cro{E_3}{E_1}}{E_4}}{E_5}}{E_2}}{E_4}}{E_3}$,\\&  $K_1K_3^2K_4^2K_5K_2$ \\

\hline $\alpha_1 + \alpha_3 + 2 \alpha_4 $& $E_1$ , $E_4$, $E_6$, $\cro{E_4}{E_2}$,  $\cro{E_4}{\cro{E_1}{E_3}}$, $\cro{E_4}{\cro{E_6}{E_5}}$, $\cro{\cro{E_4}{E_2}}{\cro{E_1}{E_3}}$, \\ $\ \ \ \ + \alpha_5 + \alpha_6 + \alpha_2$ & $\cro{\cro{E_4}{E_2}}{\cro{E_6}{E_5}}$, $\cro{\cro{E_4}{\cro{E_1}{E_3}}}{\cro{E_6}{E_5}}$, \\& $\cro{\cro{\cro{E_4}{E_2}}{\cro{E_1}{E_3}}}{\cro{E_6}{E_5}}$,  $\cro{\cro{\cro{E_4}{E_2}}{\cro{E_1}{E_3}}}{\cro{E_4}{E_5}}$, \\& $\cro{\cro{\cro{E_4}{E_2}}{\cro{E_6}{E_5}}}{\cro{E_4}{E_3}}$, $\cro{E_4}{\cro{\cro{\cro{E_4}{E_2}}{\cro{E_1}{E_3}}}{\cro{E_6}{E_5}}}$, \\& $K_1K_3K_4^2K_5K_6K_2$ \\
\hline

 $\alpha_1 + 2 \alpha_3 + 2 \alpha_4 $& $E_6$, $E_3$, $\cro{E_3}{E_1}$, $\cro{E_3}{E_4}$,
$\cro{\cro{E_3}{E_4}}{E_1}$, $\cro{\cro{E_3}{E_4}}{E_2}$,\\ $ \ \ \ \ \alpha_5 + \alpha_6 + \alpha_2$ *& $\cro{\cro{E_3}{E_4}}{\cro{E_6}{E_5}}$, $\cro{\cro{\cro{E_3}{E_4}}{E_2}}{E_1}$, $\cro{\cro{\cro{E_3}{E_4}}{E_1}}{\cro{E_6}{E_5}}$,\\& $\cro{\cro{\cro{E_3}{E_4}}{E_2}}{\cro{E_6}{E_5}}$, $\cro{\cro{\cro{\cro{E_3}{E_4}}{E_2}}{E_1}}{\cro{E_6}{E_5}}$,\\& $\cro{\cro{\cro{\cro{E_3}{E_4}}{E_2}}{\cro{E_6}{E_5}}}{E_4}$, $\cro{\cro{\cro{\cro{\cro{E_3}{E_4}}{E_2}}{E_1}}{\cro{E_6}{E_5}}}{E_4}$,\\&
$\cro{\cro{\cro{\cro{E_3}{E_4}}{E_2}}{E_1}}{\cro{\cro{E_3}{E_4}}{E_5}}$,\\ & 
$\cro{E_3}{\cro{\cro{\cro{\cro{\cro{E_3}{E_4}}{E_2}}{E_1}}{\cro{E_6}{E_5}}}{E_4}}$,\\& $K_1K_3^2K_4^2K_5K_6K_2$ \\

\hline $\alpha_1 + 2 \alpha_3 + 2 \alpha_4 $ & $E_3$, $E_5$, $\cro{E_3}{E_1}$, $\cro{E_5}{E_6}$, $\cro{E_3}{\cro{E_5}{E_4}}$, $\cro{E_3}{\cro{\cro{E_5}{E_6}}{E_4}}$,\\ $\ \ \ \ + 2 \alpha_5 + \alpha_6 + \alpha_2$ & $\cro{\cro{E_3}{E_1}}{\cro{E_5}{E_4}}$,
$\cro{\cro{E_3}{\cro{E_5}{E_4}}}{E_2}$, $\cro{\cro{E_3}{E_1}}{\cro{\cro{E_5}{E_6}}{E_4}}$,\\& $\cro{\cro{E_3}{\cro{\cro{E_5}{E_6}}{E_4}}}{E_2}$, $\cro{\cro{\cro{E_3}{E_1}}{\cro{E_5}{E_4}}}{E_2}$,\\& $\cro{\cro{\cro{E_3}{E_1}}{\cro{\cro{E_5}{E_6}}{E_4}}}{E_2}$, $\cro{\cro{\cro{\cro{E_3}{E_1}}{\cro{E_5}{E_4}}}{E_2}}{\cro{E_3}{E_4}}$,\\& $\cro{\cro{\cro{E_3}{\cro{\cro{E_5}{E_6}}{E_4}}}{E_2}}{\cro{E_5}{E_4}}$,\\ & $\cro{\cro{\cro{\cro{E_3}{E_1}}{\cro{\cro{E_5}{E_6}}{E_4}}}{E_2}}{\cro{E_3}{E_4}}$,\\& $\cro{\cro{\cro{\cro{E_3}{E_1}}{\cro{\cro{E_5}{E_6}}{E_4}}}{E_2}}{\cro{E_5}{E_4}}$,\\&
$\cro{E_5}{\cro{\cro{\cro{\cro{E_3}{E_1}}{\cro{\cro{E_5}{E_6}}{E_4}}}{E_2}}{\cro{E_3}{E_4}}}$,\\&
$K_1K_3^2K_4^2K_5^2K_6K_2$ \\

\hline
\end{tabular}}
 
\begin{flushleft}
\scriptsize{ \begin{tabular}{|l|l|}

\hline  $\alpha_1 + 2 \alpha_3 + 3 \alpha_4 $ & $E_4$, $\cro{E_4}{E_3}$, $\cro{E_4}{E_5}$,
$\cro{\cro{E_4}{E_5}}{E_6}$, $\cro{\cro{E_4}{E_3}}{E_1}$, $\cro{\cro{E_4}{E_3}}{E_5}$,\\$ \ \ \ \ + 2 \alpha_5 + \alpha_6 + \alpha_2$& $\cro{\cro{\cro{E_4}{E_3}}{E_5}}{E_1}$, $\cro{\cro{\cro{E_4}{E_3}}{E_5}}{E_6}$, $\cro{\cro{\cro{\cro{E_4}{E_3}}{E_5}}{E_1}}{E_6}$,\\& $\cro{\cro{\cro{E_4}{E_3}}{E_5}}{\cro{E_4}{E_2}}$, $\cro{\cro{\cro{\cro{E_4}{E_3}}{E_5}}{E_1}}{\cro{E_4}{E_2}}$,\\& $\cro{\cro{\cro{\cro{E_4}{E_3}}{E_5}}{E_6}}{\cro{E_4}{E_2}}$, $\cro{\cro{\cro{\cro{\cro{E_4}{E_3}}{E_5}}{E_1}}{E_6}}{\cro{E_4}{E_2}}$,\\&
$\cro{\cro{\cro{\cro{\cro{E_4}{E_3}}{E_5}}{E_1}}{\cro{E_4}{E_2}}}{E_3}$,
$\cro{\cro{\cro{\cro{\cro{E_4}{E_3}}{E_5}}{E_6}}{\cro{E_4}{E_2}}}{E_5}$,\\&
$\cro{\cro{\cro{\cro{\cro{\cro{E_4}{E_3}}{E_5}}{E_1}}{E_6}}{\cro{E_4}{E_2}}}{E_3}$,\\&
$\cro{\cro{\cro{\cro{\cro{\cro{E_4}{E_3}}{E_5}}{E_1}}{E_6}}{\cro{E_4}{E_2}}}{E_5}$,\\&
$\cro{\cro{\cro{\cro{\cro{\cro{\cro{E_4}{E_3}}{E_5}}{E_1}}{E_6}}{\cro{E_4}{E_2}}}{E_5}}{E_3}$,\\&
$\cro{E_4}{\cro{\cro{\cro{\cro{\cro{\cro{\cro{E_4}{E_3}}{E_5}}{E_1}}{E_6}}{\cro{E_4}{E_2}}}{E_5}}{E_3}}$,\\&
$K_1K_3^2K_4^3K_5^2K_6K_2$ \\

\hline  $\alpha_1 + 2 \alpha_3 + 3 \alpha_4 $ &
 $E_2$, $\cro{E_2}{E_4}$, $\cro{\cro{E_2}{E_4}}{E_5}$, $\cro{\cro{E_2}{E_4}}{E_3}$, $\cro{\cro{\cro{E_2}{E_4}}{E_3}}{E_5}$,\\ $ \ \ \ \ + 2 \alpha_5 + \alpha_6 + 2 \alpha_2$ &  $\cro{\cro{\cro{E_2}{E_4}}{E_3}}{E_1}$, $\cro{\cro{\cro{E_2}{E_4}}{E_5}}{E_6}$,
$\cro{\cro{\cro{\cro{E_2}{E_4}}{E_3}}{E_1}}{E_5}$, \\ &
$\cro{\cro{\cro{\cro{E_2}{E_4}}{E_5}}{E_6}}{E_3}$, 
$\cro{\cro{\cro{\cro{\cro{E_2}{E_4}}{E_3}}{E_1}}{E_5}}{E_6}$, \\ &
 $\cro{\cro{\cro{\cro{E_2}{E_4}}{E_3}}{E_5}}{E_4}$, $\cro{\cro{\cro{\cro{\cro{E_2}{E_4}}{E_3}}{E_1}}{E_5}}{E_4}$, \\ & $\cro{\cro{\cro{\cro{\cro{E_2}{E_4}}{E_5}}{E_6}}{E_3}}{E_4}$,  $\cro{\cro{\cro{\cro{\cro{\cro{E_2}{E_4}}{E_3}}{E_1}}{E_5}}{E_6}}{E_4}$, \\ &
 $\cro{\cro{\cro{\cro{\cro{\cro{E_2}{E_4}}{E_3}}{E_1}}{E_5}}{E_4}}{E_3}$, $\cro{\cro{\cro{\cro{\cro{\cro{E_2}{E_4}}{E_5}}{E_6}}{E_3}}{E_4}}{E_5}$, \\ &
 $\cro{\cro{\cro{\cro{\cro{\cro{\cro{E_2}{E_4}}{E_3}}{E_1}}{E_5}}{E_6}}{E_4}}{E_3}$, \\ & $\cro{\cro{\cro{\cro{\cro{\cro{\cro{E_2}{E_4}}{E_3}}{E_1}}{E_5}}{E_6}}{E_4}}{E_5}$,\\ &
 $\cro{\cro{\cro{\cro{\cro{\cro{\cro{\cro{E_2}{E_4}}{E_3}}{E_1}}{E_5}}{E_6}}{E_4}}{E_3}}{E_5}$,\\ &
 $\cro{\cro{\cro{\cro{\cro{\cro{\cro{\cro{\cro{E_2}{E_4}}{E_3}}{E_1}}{E_5}}{E_6}}{E_4}}{E_3}}{E_5}}{E_4}$,\\ &
 $\cro{E_2}{\cro{\cro{\cro{\cro{\cro{\cro{\cro{\cro{\cro{E_2}{E_4}}{E_3}}{E_1}}{E_5}}{E_6}}{E_4}}{E_3}}{E_5}}{E_4}}$,\\ &
$K_1K_3^2K_4^3K_5^2K_6K_2$ \\
\hline
 \end{tabular} }
 \end{flushleft}

\bibliography{biblio}

\end{document}